\documentclass[11pt]{amsart}

\usepackage{latexsym, fancyhdr, enumitem}
\usepackage{amsmath, amsfonts, amssymb}
\usepackage[english]{babel}
\usepackage{graphicx}
\usepackage[breaklinks=true]{hyperref}

\usepackage{fullpage}

\newtheorem{prop}{Proposition}[section]
\newtheorem{lem}[prop]{Lemma}
\newtheorem{fol}[prop]{Corollary}
\newtheorem{theo}[prop]{Theorem}

\theoremstyle{definition}
\newtheorem{defi}[prop]{Definition}

\theoremstyle{remark}
\newtheorem{rem}[prop]{Remark}
\newtheorem{remdef}[prop]{Definition and Remark}
\newtheorem{exmp}[prop]{Example}

\newcommand{\field}[1]{\mathbb{#1}}

\newcommand{\ZZ}{\field{Z}}

\newcommand{\RR}{\field{R}}
\newcommand{\NN}{\field{N}}

\newcommand{\KKK}{\mathcal{K}}
\newcommand{\PP}{\mathcal{P}}

\newcommand{\SSSS}{\mathcal{S}}
\newcommand{\AAA}{\mathcal{A}}

\DeclareMathOperator*{\codim}{\mathrm{codim}}

\DeclareMathOperator{\id}{\mathrm{id}}

\DeclareMathOperator{\supp}{\mathrm{supp}}

\DeclareMathOperator*{\Hom}{\mathrm{Hom}}

\DeclareMathOperator*{\GL}{\mathrm{GL}}

\DeclareMathOperator*{\Aut}{\mathrm{Aut}}

\DeclareMathOperator*{\Sym}{\mathrm{Sym}}

\DeclareMathOperator*{\Rad}{\mathrm{Rad}}

\DeclareMathOperator*{\Cham}{\mathrm{Cham}}

\newcommand{\ol}[1]{\overline{#1}}

\newcommand{\re }{^\mathrm{re}}
\newcommand{\rer }[1]{(R\re)^{#1}}
\newcommand{\rsC }{\mathcal{R}}
\newcommand{\rfl }{\rho }
\newcommand{\Cm }{C}

\newcommand{\Cc }{\mathcal{C}}
\newcommand{\cm }{c}
\newcommand{\s }{\sigma }
\newcommand{\Wg }{\mathcal{W}}
\newcommand{\Ob }{\mathrm{Ob}}

\title[On the Tits cone of a Weyl groupoid]
{On the Tits cone of a Weyl groupoid}

\author{M.~Cuntz}
\address{Michael Cuntz,
Institut f\"ur Algebra, Zahlentheorie und Diskrete Mathematik,
Fakult\"at f\"ur Mathematik und Physik,
Leibniz Universit\"at Hannover,
Welfengarten 1,
D-30167 Hannover, Germany}
\email{cuntz@math.uni-hannover.de}

\author{B.~M\"uhlherr}
\address{Bernhard M\"uhlherr,
Mathematisches Institut, Arndtstra{\ss}e 2, 35392 Gie{\ss}en, Germany}
\email{bernhard.muehlherr@math.uni-giessen.de}

\author{C.~J.~Weigel}
\address{Christian J. Weigel,
Mathematisches Institut, Arndtstra{\ss}e 2, 35392 Gie{\ss}en, Germany}
\email{christian.j.weigel@math.uni-giessen.de}

\begin{document}

\begin{abstract}
We translate the axioms of a Weyl groupoid with (not necessarily finite) root system in terms of arrangements.
The result is a correspondence between Weyl groupoids permitting a root system and Tits arrangements satisfying an integrality condition which we call the crystallographic property.
\end{abstract}

\maketitle

\section{Introduction}

A large part of the (yet achieved) classification of finite dimensional Nichols algebras relies on a symmetry structure called the Weyl groupoid.
This was constructed first in \cite{p-H-06} for Nichols algebras of diagonal type, and then in \cite{p-AHS-08} in a very general setting.
Recently Heckenberger and Vendramin \cite{MR3605018}, \cite{MR3656477} exploited the Weyl groupoid of a Nichols algebra further and classified all finite dimensional Nichols algebras of semisimple Yetter-Drinfeld modules of rank greater than one over finite nonabelian groups, thus in other words those Nichols algebras of interest for the classification of pointed Hopf algebras which define a non trivial Weyl groupoid.

The fact that a Nichols algebra is finite dimensional translates into the finiteness of the sets of real roots of the corresponding Weyl grou\-poid. These `finite Weyl groupoids' were completely classified in a series of papers by Heckenberger and the first author culminating in \cite{p-CH10}.
If one is interested in classifying arbitrary Nichols algebras defining a Weyl groupoid with root system, it is thus very natural to understand Weyl groupoids admitting a root system in general, i.e.\ with possibly infinitely many roots.

In the course of the classification of finite Weyl groupoids it was observed that they correspond to so-called crystallographic arrangements, which was finally proven in \cite{Cu11}.
The main purpose of this paper is to generalize the result of \cite{Cu11} from finite to arbitrary Weyl groupoids permitting a root system:
a finite Weyl groupoid defines a (finite) simplicial arrangement, an arbitrary Weyl groupoid (with root system) defines a so-called Tits arrangement, see \cite{CMW} for a definition.
The integrality property called ``crystallographic'' in \cite{Cu11} may be transfered to Tits arrangements without trouble although one has to be very careful with the details. Thus our main result is divided into two parts: every crystallographic arrangement defines a Weyl groupoid (see Section \ref{sec_cryprop}), and every Weyl groupoid (with root system) defines a crystallographic arrangement (see Section \ref{geom.real}).

\begin{theo}[Cor.\ \ref{cor:cor}]
There exists a one-to-one correspondence between connected, simply connected Cartan graphs permitting a root system and crystallographic Tits arrangements with reduced root system.

Under this correspondence, equivalent Cartan graphs correspond to combinatorially equivalent Tits arrangements and vice versa, giving rise to a one-to-one correspondence between the respective equivalence classes.
\end{theo}

As a result, most of the Nichols algebras define Tits cones. An approach to classify arbitrary Nichols algebras could now be to start with those algebras defining a nice cone. For example, if the Tits cone is a halfspace, we call the Nichols algebra `affine'.
A first classification result of affine Nichols algebras of diagonal type is \cite{C16}.

This paper is organized as follows.
Section \ref{sec_titsarr} recalls the relevant notions introduced in \cite{CMW}.
In Section \ref{sec_cryprop}, we discuss the crystallographic property and how to obtain a Weyl groupoid from a crystallographic arrangement. The geometric realization of a connected simply connected Weyl groupoid is given in Section \ref{geom.real}. Section \ref{subarr} discusses arrangements induced by crystallographic arrangements by subspaces.

\medskip

\textbf{Acknowledgement.}
Some of the results and ideas of this note were achieved during a Mini-Workshop on Nichols algebras and Weyl groupoids at the Mathematisches Forschungsinstitut Oberwolfach in October 2012, and during meetings in Gie{\ss}en, Hannover, and Kaiserslautern supported by the Deutsche Forschungsgemeinschaft within the priority programme 1388.

\section{Tits arrangements}\label{sec_titsarr}

\subsection{Hyperplane arrangements}

In this section we recall the definitions and properties regarding hyperplane arrangements and Tits arrangements. For some basic examples and the proofs of the statements see \cite{CMW}. Our notation follows this paper as well.

Note that all topological notations we use refer to the standard topology in $\RR^r$, in particular for $X \subset \RR^r$ we denote by $\ol{X}$ the closure of $X$.

\begin{defi}
Let $\AAA$ be a set of linear hyperplanes in $V = \RR^r$, and $T$ an open convex cone in $V$. We say that \textit{$\AAA$ is locally finite in $T$} if for every $x \in T$ there exists a neighbourhood $U_x \subset T$ of $x$, such that $\{H \in \AAA \mid H \cap U_x \neq \emptyset\}$ is a finite set.

A \textit{hyperplane arrangement (of rank $r$)} is a pair $(\AAA, T)$, where $T$ is a convex open cone in $V$, and $\AAA$ is a (possibly infinite) set of linear hyperplanes such that 
\begin{enumerate}
	\item $H \cap T \neq \emptyset$ for all $H \in \AAA$,
	\item $\AAA$ is locally finite in $T$.
\end{enumerate}
If $T$ is unambiguous from the context, we also call the set $\AAA$ a hyperplane arrangement.

Let $X \subset \ol{T}$. Then the \textit{localisation at $X$ (in $\AAA$)} is defined as
$$\AAA_X := \{H \in \AAA \mid X \subset H\}.$$
If $X= \{x\}$ is a singleton, we write $\AAA_x$ instead of $\AAA_{\{x\}}$ and call $(\AAA_x,T)$ the \textit{parabolic subarrangement at $x$}. A \textit{parabolic subarrangement of $(\AAA,T)$} is a subarrangement $(\AAA',T)$, such that $\AAA' = \AAA_x$ for some $x \in \ol{T}$. For the purpose of this paper we call the set $$\sec_\AAA(X) := \bigcup_{x \in X} \AAA_x = \{H \in \AAA \mid H \cap X \neq \emptyset\}$$ the \textit{section of $X$ (in $\AAA$)}. We will omit $\AAA$ when there is no danger of confusion.

The \textit{support of $X$ (in $\AAA$)} is defined to be the subspace $\supp_\AAA(X) = \bigcap_{H \in \AAA_X} H$.

The connected components of $T \setminus \bigcup_{H \in \AAA} H$ are called \textit{chambers}, denoted with $\mathcal{K}(\AAA,T)$ or just $\mathcal{K}$, if $(\AAA,T)$ is unambiguous.

Let $K \in \KKK(\AAA,T)$. Define the \textit{walls of $K$} to be the elements of
$$W^K := \{H \leq V \mid H \text{ hyperplane, } \langle H \cap \overline{K} \rangle = H, H \cap K^\circ = \emptyset\}.$$

The \textit{radical of $\AAA$} is the subspace $\Rad(\AAA) := \bigcap_{H \in \AAA} H = \supp_\AAA(0)$. We call the arrangement \textit{non-degenerate} if $\Rad(\AAA) = 0$, and \textit{degenerate} otherwise. A hyperplane arrangement is \textit{thin} if $W^K \subset \AAA$ for all $K \in \KKK$.
\end{defi}

We recall a basic consequence of the notion of local finiteness.

\begin{lem}[{\cite[Lemma 2.3]{CMW}}]
Let $(\AAA, T)$ be a hyperplane arrangement. Then for every point $x \in T$ there exists a neighbourhood $U_x$ such that $\AAA_x = \sec(U_x)$.
Furthermore the set $\sec(X)$ is finite for every compact set $X \subset T$.
\label{compfin}
\end{lem}

\subsection{Simplicial cones and Tits arrangements}

\begin{defi}
In the following let $V = \RR^r$, $r \geq 1$. For a linear form $\alpha \in V^\ast$, define
\begin{align*}
\alpha^\perp &:= \ker \alpha,\\
\alpha^+ &:= \alpha^{-1}(\RR_{>0}),\\
\alpha^- &:= \alpha^{-1}(\RR_{<0}).
\end{align*}
Let $B$ be a basis of $V^\ast$, then the \textit{open simplicial cone (associated to $B$)} is 
$$
K^B := \bigcap_{\alpha \in B} \alpha^+.
$$
\end{defi}

\begin{remdef}
With notation as above we find
\begin{align*}
\ol{\alpha^+} &= \alpha^{-1}(\RR_{\geq 0}) = \alpha^\perp \cup \alpha^+,\\
\ol{\alpha^-} &= \alpha^{-1}(\RR_{\leq 0}) = \alpha^\perp \cup \alpha^-.
\end{align*}
Let $B$ be a basis of $V^\ast$. We can then define the \textit{closed simplicial cone (associated to $B$)} as 
$$
\ol{K^B} = \bigcap_{\alpha \in B}\ol{\alpha^+}.
$$

We call a cone \textit{simplicial} if it is open simplicial or closed simplicial.

A simplicial cone can also be defined using bases of $V$. Let $C$ be a basis of $V$, then the open simplicial cone associated to $C$ is 
$$
K^C = \{\sum_{v \in C}\lambda_v v \mid \lambda_v >0\}
$$
and the closed simplicial cone associated to $C$ is 
$$
\ol{K^C} = \{\sum_{v \in C}\lambda_v v \mid \lambda_v \geq 0\}.
$$

Both concepts are equivalent, and it is immediate from the definition that if $B \subset V^\ast$ and $C \subset V$ are bases, then $K^B = K^C$ if and only if $B$ is, up to positive scalar multiples and permutation, dual to $C$.

A simplicial cone $K$ associated to $B$ carries a natural structure of a simplex, to be precise:

$$\SSSS_K := \{\ol{K} \cap \bigcap_{\alpha \in B'} \alpha^\perp \mid B' \subset B\}$$
is a poset with respect to inclusion, which is isomorphic to $\PP(B)$ with inverse inclusion. If $C$ is the basis of $V$ dual to $B$, we find $\SSSS_K$ to be the set of all convex combinations of subsets of $C$, and $\SSSS_K$ is also isomorphic to $\PP(C)$. Moreover, $\{\RR_{\geq 0}c \mid c \in C\}$ is the vertex set of the simplex $\SSSS_K$.

For a simplicial cone $K$, we denote with $B_K \subset V^\ast$ a basis of $V^\ast$ such that $K^{B_K} = K$. The basis $B_K$ is uniquely determined by $K$ up to permutation and positive scalar multiples.
\label{rem:scs}
\end{remdef}

\begin{lem}
Let $K = K^C = K^B$ for a basis $C$ of $V$ and a basis $B$ of $V^\ast$. Let $\beta \in V^\ast$. Then $\beta(v) \geq 0$ for all $v \in C$ if and only if $\beta \in \sum_{\alpha \in B} \RR_{\geq 0} \alpha$. Likewise $\beta(v) \leq 0$ for all $v \in C$ if and only if $\beta \in -\sum_{\alpha \in B} \RR_{\geq 0} \alpha$.\label{purelc}
\end{lem}

\begin{defi}
Let $T \subseteq V$ be a convex open cone and $\AAA$ a set of linear hyperplanes in $V$. We call a hyperplane arrangement $(\AAA, T)$ a \textit{simplicial arrangement (of rank $r$)}, if every $K \in \mathcal{K}(\AAA)$ is an open simplicial cone.

The cone $T$ is the \textit{Tits cone} of the arrangement. A simplicial arrangement is a \textit{Tits arrangement} if it is thin.
\end{defi}

\begin{defi}
A simplicial arrangement $(\AAA,T)$ is \textit{spherical} if $T = \RR^r$. We say $(\AAA,T)$ is \textit{affine} if $T = \gamma^+$ for some $0 \neq \gamma \in V^\ast$. For an affine arrangement we call $\gamma$ an \textit{imaginary root} of the arrangement.
\end{defi}

\subsection{Root systems}

\begin{defi}
Let $V = \RR^r$, a \textit{root system} is a set $R \subset V^\ast$ such that
\begin{enumerate}[label=\arabic*)]
  \item $0 \notin R$,
	\item $-\alpha \in R$ for all $\alpha \in R$,
	\item there exists a Tits arrangement $(\AAA,T)$ such that $\AAA = \{\alpha^\perp \mid \alpha \in R\}$. 
\end{enumerate}
If $R$ is a root system and $(\AAA,T)$ as in 3), we say that the Tits arrangement $(\AAA, T)$ is \textit{associated to the root system $R$}.
We call a root system $R$ \textit{reduced} if $R \cap \langle \alpha \rangle = \{\pm\alpha\}$ for all $\alpha\in R$.
\label{def:rs}
\end{defi}

\begin{rem}
Most of the time we consider reduced root systems in this paper. However, root systems which are not reduced appear naturally when considering restrictions in Section \ref{subarr}.
\end{rem}

\begin{defi}
Let $(\AAA, T)$ be a Tits arrangement associated to $R$. Let $K$ be a chamber. The \textit{root basis of $K$} is the set 
$$B^K := \{\alpha \in R \mid \alpha^\perp \in W^K, \alpha(x) > 0 \text{ for all } x \in K\}.$$
\end{defi}

\begin{rem}
If $(A,T)$ is a Tits arrangement associated to $R$ and $K \in \KKK$, then
$$
W^K = \{\alpha^\perp \mid \alpha \in B^K\}.
$$

Also, as a simplicial cone $K \subset \RR^r$ has exactly $r$ walls, the set $B^K$ is a basis of $V^\ast$. Notice that $K^{B^K} = K$.
\end{rem}

\begin{lem}[{\cite[Lemma 3.16]{CMW}}]
Let $(\AAA,T)$ be a Tits arrangement associated to $R$, $K$ a chamber. Then $R \subset \pm \sum_{\alpha \in B^K} \RR_{\geq 0}\alpha$. In other words, every root is a non-negative or non-positive linear combination of $B^K$.
\label{posorneg}
\end{lem}

\begin{defi}
Let $(\AAA,T)$, $(\AAA', T')$ be Tits arrangements associated to $R$,$R'$. Then $(\AAA,T)$ and $(\AAA',T')$ are called \textit{combinatorially equivalent} if there exists an $g \in \GL(V)$ such that $g(\AAA) = \AAA'$, $g\ast R = R'$, $g(T) = T'$. Here $\ast$ denotes the dual action of $GL(V)$ on $V^\ast$, defined by $g \ast \alpha = \alpha \circ g^{-1}$.
\end{defi}

\begin{remdef}
Let $(\AAA, T)$ be a simplicial hyperplane arrangement. The set of chambers $\mathcal{K}$ gives rise to a poset
$$
\mathcal{S}(\AAA,T) := \left\{\ol{K} \cap \bigcap_{H \in \AAA'}H \mid K \in \KKK, \AAA' \subseteq W^K\right\} = \bigcup_{K \in \KKK}\SSSS_K,$$
with set-wise inclusion giving a poset-structure. Note that we do not require any of these intersections to be in $T$. By construction they are contained in the closure of $T$, as every $K$ is an open subset in $T$. We write $\mathcal{S}$ instead of $\mathcal{S}(\AAA,T)$ if $(\AAA, T)$ is uniquely determined from the context.
\end{remdef}

We recall the properties of $\SSSS$ from \cite{CMW}. Note that we will not recall definitions and properties of simplicial complexes in this paper. As a reference please consult the Appendix of \cite{CMW}.

\begin{prop}[{\cite[Proposition 3.26]{CMW}}]
Let $(\AAA,T)$ be a simplicial arrangement. The complex $\SSSS: = \SSSS(\AAA,T)$ is a chamber complex of rank $r$ with
$$\Cham(\SSSS) = \{\ol{K} \mid K \in \KKK\}.$$
The complex $\SSSS$ is gated and strongly connected. Furthermore there exists a type function $\tau: \SSSS \to I$ of $\SSSS$, where $I = \{1, \dots, r\}$. The complex $\SSSS$ is thin if and only if $(\AAA, T)$ is thin, and $\SSSS$ is spherical if and only if $(\AAA,T)$ is spherical.
\label{prop-Sproperties}
\end{prop}

\begin{defi}
A simplicial arrangement $(\AAA,T)$ of rank $r$ is called \textit{$k$-spherical} for $k \in \NN_0$ if every simplex $S$ of $\mathcal{S}$, such that $\codim(S) = k$, meets $T$. We say $(\AAA, T)$ is \textit{locally spherical} if it is $r-1$-spherical.
\label{def:spherical}
\end{defi}

\begin{rem}
An equivalent condition for $(\AAA,T)$ to be $k$-spherical, which we will use often, is that every $(r-k-1)$-simplex meets $T$. This uses just the fact that simplices of codimension $k$ are exactly $(r-k-1)$-simplices.
\end{rem}

\section{The crystallographic property}\label{sec_cryprop}

\subsection{Crystallographic arrangements}

With respect to Lemma \ref{posorneg} we can make the following definition:

\begin{defi}
Let $(\AAA,T)$ be a Tits arrangement associated to $R$. We call $(\AAA,T)$ \textit{crystallographic (with respect to $R$)} if it satisfies
$$
R \subset \pm\sum_{\alpha \in B^K} \NN_0 \alpha
$$
for all $K \in \KKK$.
\label{defcry}
\end{defi}

From now on, let $(\AAA,T)$ be a crystallographic Tits arrangement with respect to $R$.

We will now take a closer look at the relations between the bases of adjacent chambers. The proof of the following lemma is exactly as in the spherical case.

\begin{lem}[{\cite[Lemma 2.8]{Cu11}}]
Let $K,L \in \KKK$ be adjacent chambers. Assume $\ol{K} \cap \ol{L} \subset \alpha_1^\perp$ for $\alpha_1 \in B^K$.
Then
\[ B^L \subseteq \{-\alpha_1\} \cup \sum_{\alpha \in B^K}\NN_0 \alpha. \]
\label{adjcham1}
\end{lem}

We recall the notion of compatible indexing of root bases, compare \cite[Def.\ and Rem.\ 3.28]{CMW}.

\begin{remdef}
Assume for $K \in \KKK$ that $B^K$ is indexed in some way, i.e. $B^K = \{\alpha_1, \dots, \alpha_r\}$.
For any set $I$, define the map $\kappa_I: \mathcal{P}(I) \to \mathcal{P}(I)$ by $\kappa(J) = I \setminus J$. Set $\kappa := \kappa_{\{1, \dots, r\}}$.

For every simplex $F \subset \ol{K}$ there exists a description of the form $F = \ol{K} \cap \bigcap_{\alpha \in B_F}\alpha^\perp$ for some $B_F \subset B^K$ by Remark \ref{rem:scs}, which gives an index set $J_F = \{i \mid \alpha_i \in B_F\}$.

This gives rise to a type function of $\ol{K}$ in $\mathcal{S}$, by taking the map $\tau_{\ol{K}}: F \mapsto \kappa(J_F)$. By \cite[Theorem B15]{CMW} the map $\tau_{\ol{F}}$ yields a unique type function $\tau$ of the whole simplicial complex $\SSSS$. So let $L \in \KKK$ be another chamber, then the restriction $\tau|_{\ol{L}}$ is a type function of $\ol{L}$ as well. Assume $B^L = \{\beta_1, \dots, \beta_r\}$, this yields a second type function of $\ol{L}$ in the same way we acquired a type function of $\ol{K}$ before,
$$\tau_{\ol{L}}: F \mapsto \kappa(\{i \mid F \subset \beta_i^\perp\}).
$$
We now call the indexing of $B^L$ \textit{compatible with $B^K$}, if $\tau_{\ol{L}} = \tau|_{\ol{L}}$.

Since the type function $\tau$ is unique, there is a unique indexing of $B^L$ compatible with $B^K$. 
\end{remdef}

Proposition \ref{adjcham2} can also be found in \cite{Cu11}. The argument used there can be modified by the following Lemma, which makes use of compatibility and can be found in \cite{CMW}.

\begin{lem}[{\cite[Lemma 3.31]{CMW}}]
Assume that $(\AAA, T)$ is a Tits arrangement associated to $R$. Let $K, L \in \KKK$ be adjacent chambers and choose an indexing $B^K = \{\alpha_1, \dots, \alpha_r\}$. Let the indexing of $B^L = \{\beta_1, \dots, \beta_r\}$ be compatible with $B^K$. Assume $\ol{K} \cap \ol{L} \subset \alpha_k^\perp$ for some $1 \leq k \leq r$. Then $\beta_i \in \langle \alpha_i, \alpha_k \rangle$.
\label{adjcham15}
\end{lem}

\begin{prop}
Let $(\AAA, T)$ be a crystallographic arrangement with respect to $R$ and let $K,L \in \KKK$ be adjacent chambers. Choose an indexing $B^K = \{\alpha_1, \dots, \alpha_r\}$ and let the indexing of $B^L = \{\beta_1, \dots, \beta_r\}$ be compatible with $B^K$. Assume $\ol{K} \cap \ol{L} \subset \alpha_k^\perp$ for some $1 \leq k \leq r$.

Then there exist $c_i \in \ZZ$ for $i = 1, \dots, r$ such that $\beta_i = c_i\alpha_k + \alpha_i$. Furthermore, $c_k = -2$ and $c_i \in \NN_0$ for $i \neq k$.
\label{adjcham2}
\end{prop}
\begin{proof}
Without loss of generality we can assume $k=1$, then $\beta_1 = -\alpha_1 = \alpha_1 - 2\alpha_1$. Consider the linear transformation $\sigma$, mapping $\alpha_i$ to $\beta_i$. This is an element in $\GL_r(\RR)$ with entries in $\ZZ$, since the arrangement is crystallographic and $B^K$, $B^L$ are bases. By symmetry the inverse has also entries in $\ZZ$, so $\sigma \in \GL_r(\ZZ)$ holds. Hence the matrix of $\sigma$ with respect to bases $B^K$ and $B^L$ is
$$
\begin{pmatrix}
	-1 & c_2 & \dots & c_r \\
	0 &&&\\
	\vdots & &A&\\
	0 &&&
\end{pmatrix},
$$
where by Lemma \ref{adjcham1} the $c_i$ are in $\NN$ and $A \in GL_{r-1}(\ZZ)$ with nonnegative entries. Since the matrix of $\sigma^{-1}$ is of the same form, $A^{-1} \in \GL_{r-1}(\ZZ)$ with nonnegative entries. It is well known (see Theorem 4.6 in Chapter 3 in \cite{BP79} for example) that this implies that $A$ is monomial, and since its entries are in $\ZZ$, $A$ is a permutation matrix. We know therefore that $\beta_i = \sigma(\alpha_i) = c_i\alpha_1 + \alpha_{\pi(i)}$ for some permutation $\pi$. It remains to show that $A$ is in fact the identity matrix. This is a consequence of Lemma \ref{adjcham15}, as $\beta_k = \lambda_1\alpha_1 + \lambda_k \alpha k$ for $\lambda_1, \lambda_k \in \RR$. Using that $B^K$, $B^L$ are bases we find $\pi = \id_{\{1, \dots, r\}}$.
\end{proof}

\subsection{Cartan graphs and Weyl groupoids}\label{CSWG}

We recall the notion of a Weyl groupoid which was introduced by Heckenberger and Yamane \cite{HY08} and reformulated in \cite{p-CH09a}.

\begin{defi}
Let $I:=\{1,\ldots,r\}$ and
$\{\alpha_i\,|\,i\in I\}$ the standard basis of $\ZZ ^I$.
A {\it generalized Cartan matrix}
$\Cm =(\cm _{ij})_{i,j\in I}$
is a matrix in $\ZZ ^{I\times I}$ such that
\begin{enumerate}
\item[(M1)] $\cm _{ii}=2$ and $\cm _{jk}\le 0$ for all $i,j,k\in I$ with
  $j\not=k$,
\item[(M2)] if $i,j\in I$ and $\cm _{ij}=0$, then $\cm _{ji}=0$.
\end{enumerate}
\end{defi}

\begin{defi}
Let $A$ be a non-empty set, $\rfl _i : A \to A$ a map for all $i\in I$,
and $\Cm ^a=(\cm ^a_{jk})_{j,k \in I}$ a generalized Cartan matrix
in $\ZZ ^{I \times I}$ for all $a\in A$. The quadruple
\[ \Cc = \Cc (I,A,(\rfl _i)_{i \in I}, (\Cm ^a)_{a \in A})\]
is called a \textit{Cartan graph} if
\begin{enumerate}
\item[(C1)] $\rfl _i^2 = \id$ for all $i \in I$,
\item[(C2)] $\cm ^a_{ij} = \cm ^{\rfl _i(a)}_{ij}$ for all $a\in A$ and
  $i,j\in I$.
\end{enumerate}

\end{defi}

\begin{defi}
Let $\Cc = \Cc (I,A,(\rfl _i)_{i \in I}, (\Cm ^a)_{a \in A})$ be a
Cartan graph. For all $i \in I$ and $a \in A$ define $\s _i^a \in
\Aut(\ZZ ^I)$ by
\begin{align}
\s _i^a (\alpha_j) = \alpha_j - \cm _{ij}^a \alpha_i \qquad
\text{for all $j \in I$.}
\label{fwg_eq:sia}
\end{align}
The \textit{Weyl groupoid of} $\Cc $
is the category $\Wg (\Cc )$ such that $\Ob (\Wg (\Cc ))=A$ and
the morphisms are compositions of maps
$\s _i^a$ with $i\in I$ and $a\in A$,
where $\s _i^a$ is considered as an element in $\Hom (a,\rfl _i(a))$.
The cardinality of $I$ is the \textit{rank of} $\Wg (\Cc )$.
\end{defi}
\begin{defi}
A Cartan graph is called \textit{standard} if $C^a = C^b$ for all $a,b \in A$.
A Cartan graph is called \textit{connected} if its Weyl grou\-poid
is connected, that is if for all $a,b\in A$ there exists $w\in \Hom (a,b)$.
The Cartan graph is called \textit{simply connected}
if $\Hom (a,a)=\{\id ^a\}$ for all $a\in A$.

Two Cartan graphs
$$\Cc = \Cc (I,A,(\rfl _i)_{i \in I}, (\Cm ^a)_{a \in A}), \quad \Cc' = \Cc (I',A',(\rfl' _i)_{i \in I}, ({\Cm'} ^a)_{a \in A'})$$
are called \textit{equivalent} if there are bijections $\varphi_0: I \to I'$, $\varphi_1: A \to A'$ such that $\varphi_1(\rfl_i(a)) = \rfl'_{\varphi_0(i)}(\varphi_1(a))$ and $c_{\varphi_0(i)\varphi_0(j)}^{\varphi_1(a)} = c_{ij}^a$ for all $i,j \in I$, $a \in A$.
\end{defi}

Let $\Cc $ be a Cartan graph. For all $a\in A$ let
\[ \rer a=\{ \id ^a \s _{i_1}\cdots \s_{i_k}(\alpha_j)\,|\,
k\in \NN _0,\,i_1,\dots,i_k,j\in I\}\subseteq \ZZ ^I.\]
The elements of the set $\rer a$ are called \emph{real roots} (at $a$).
The pair $(\Cc ,(\rer a)_{a\in A})$ is denoted by $\rsC \re (\Cc )$.
A real root $\alpha\in \rer a$, where $a\in A$, is called positive
(resp.\ negative) if $\alpha\in \NN _0^I$ (resp.\ $\alpha\in -\NN _0^I$).

\begin{defi}
Let $\Cc =\Cc (I,A,(\rfl _i)_{i\in I},(\Cm ^a)_{a\in A})$ be a Cartan
graph. For all $a\in A$ let $R^a\subseteq \ZZ ^I$, and define
$m_{i,j}^a= |R^a \cap (\NN_0 \alpha_i + \NN_0 \alpha_j)|$ for all $i,j\in
I$ and $a\in A$. We say that
\[ \rsC = \rsC (\Cc , (R^a)_{a\in A}) \]
is a \textit{root system of type} $\Cc$ if it satisfies the following
axioms.
\begin{enumerate}
\item[(R1)]
$R^a=R^a_+\cup - R^a_+$, where $R^a_+=R^a\cap \NN_0^I$, for all
$a\in A$.
\item[(R2)]
$R^a\cap \ZZ\alpha_i=\{\alpha_i,-\alpha_i\}$ for all $i\in I$, $a\in A$.
\item[(R3)]
$\s _i^a(R^a) = R^{\rfl _i(a)}$ for all $i\in I$, $a\in A$.
\item[(R4)]
If $i,j\in I$ and $a\in A$ such that $i\not=j$ and $m_{i,j}^a$ is
finite, then
$(\rfl _i\rfl _j)^{m_{i,j}^a}(a)=a$.
\end{enumerate}
\end{defi}

The root system $\rsC $ is called \textit{finite} if for all $a\in A$ the
set $R^a$ is finite. By \cite[Prop.\,2.12]{p-CH09a},
if $\rsC $ is a finite root system
of type $\Cc $, then $\rsC =\rsC \re $, and hence $\rsC \re $ is a root
system of type $\Cc $ in that case.
Roots which are not real roots are called \emph{imaginary roots}.

\begin{rem}
If $\Cc $ is a Cartan graph and there exists a root system of type $\Cc $,
then $\Cc $ satisfies
\begin{itemize}
  \item [(C3)] If $a,b\in A$ and $\id \in \Hom (a,b)$, then $a=b$.
\end{itemize}
\end{rem}

\begin{defi}
Let $\Cc =\Cc (I,A,(\rfl _i)_{i\in I},(\Cm ^a)_{a\in A})$ be a Cartan graph with Weyl group\-oid $\Wg (\Cc )$. Let $b \in A$, $J \subseteq I$ and $\Pi_J = \langle \rfl_j \mid j \in J \rangle \leq \Sym(A)$. For $a \in A$, let $\Cm_J^a$ be the restriction of $\Cm^a$ to the indices $J$, i.e.\ $\Cm_J^a := (c_{ij}^a)_{i,j \in J}$. Define the \textit{$J$-residue of $\Cc$ (containing $b$)} as
$$\Cc^b_J := (J,\Pi_J(b), (\rfl _j)_{j\in J}, (\Cm_J^a)_{a \in \Pi_J(b)}).
$$
\end{defi}

\begin{lem}
With notation as above, $\Cc_J^b$ is connected. If $\Cc$ is simply connected, then so is $\Cc_J^b$. If $\Cc$ admits a root system, then so does $\Cc_J^b$, and if furthermore $\Pi_J(b)$ is finite, $\Cc_J^b$ admits a finite real root system.
\label{residues}
\end{lem}
\begin{proof}
The Cartan graph $\Cc_J^b$ is connected since $\Pi_J$ is transitive on $\Pi_J(b)$. Since the morphisms of $\Wg(\Cc_J^b)$ are generated by a subset of those generating $\Wg(\Cc)$, $\Cc_J^b$ is simply connected if $\Cc$ is.

If $\rsC$ is a root system for $\Cc$ , then for $a \in \Pi_J(b)$ the set
$$
\rer a=\{ \id ^a \s _{i_1}\cdots \s_{i_k}(\alpha_j)\,|\,
k\in \NN _0,\,i_1,\dots,i_k,j\in J\}
$$
is a subset of the real roots at $a$ in $\Cc$, hence $\rer a$ satisfies (R1) - (R4). Assume that $\Pi_J(b)$ is finite, then for all $a \in \Pi_J(b)$ the set $\rer a$ contains at most $|J||\Pi_J(b)|$ elements, and is hence finite.
\end{proof}

\subsection{Cartan graphs and crystallographic arrangements}

In the following section we point out one part of the correspondence between reduced root systems of crystallographic Tits arrangements and real roots of a Cartan graph. We will use some results for subarrangements, which will be provided in Section \ref{subarr}.

Proposition \ref{adjcham2} allows us to make the following definition:

\begin{defi}
Let $(\AAA, T)$ be a crystallographic Tits arrangement with respect to $R$. Let $K$ be a chamber and $B^K = \{\alpha_1, \dots, \alpha_r\}$. Let $K^i$ be the chamber $i$-adjacent to $K$, i.e.\ $\ol{K} \cap \ol{K^i} \subset \alpha_i^\perp$. Let $B^{K^i} = \{\beta_1^i, \dots, \beta_r^i\}$ be indexed compatibly, then by Proposition \ref{adjcham2} we find that $\beta_j^i = c_j^i\alpha_i + \alpha_j$ with $c_j^i \in \NN$ for $i \neq j$ and $c_i^i = -2$. We will call the matrix $C^K  = (-c_j^i)_{1\leq i,j \leq r}$ the \textit{Cartan matrix at $K$}.

Furthermore, in the above setting we denote by $\varphi_{K^i,K}$ the linear extension of the map $\alpha_j \mapsto \beta_j^i$ for all $j = 1, \dots, r$.
\end{defi}

The following lemma shows that the notion of a Cartan matrix at a chamber $K$ is justified, as it is indeed a generalized Cartan matrix. For the sake of brevity, we omit the fact that it should be called the ``generalized Cartan matrix at a chamber''. In this work we will not define what a (non-generalized) Cartan matrix is.

\begin{lem}
Let $K \in \KKK$, $C^K$ the Cartan matrix at $K$. Then $C^K$ is a generalized Cartan matrix.
\label{CKGCM}
\end{lem}
\begin{proof}
The matrix $C^K$ satisfies (M1) from the definition by Proposition \ref{adjcham2}. So assume $c_i^j = 0$. This implies $\beta_i^j = \alpha_i$. By construction $\beta_j^i = c_j^i\alpha_i + \alpha_j$. Assume $c_j^i > 0$ and let $v = -c_j^i\alpha_j^\vee + \alpha_i^\vee \in V$. Then $\beta_i^j(v) = \alpha_i(v) = 1$, $\beta_j^j(v) = -\alpha_j(v) = c_j^i$ and therefore $\beta_j^i(v) = 0$. The last equality means $v \in (\beta_j^i)^\perp$, which contradicts the simplicial structure of $\mathcal{S}$. So (M2) holds.
\end{proof}

\begin{prop}
Let $(\AAA, T)$ be a crystallographic Tits arrangement with respect to $R$ and assume the $B^K$ are indexed compatibly for all $K\in \KKK$. Set $I := \{1, \dots, r\}$, $A := \KKK$, $C^K$ the generalized Cartan matrix at $K$, and for $i \in I$ let $\rho_i: A \to A$, $K \mapsto K^i$, where $K^i$ is the chamber $i$-adjacent to $K$.
Then $\Cc := \Cc(I, A, (\rho_i)_{i\in I}, (C^a)_{a \in A})$ is a connected Cartan graph.
\label{CAtoCS}
\end{prop}
\begin{proof}
By Lemma \ref{CKGCM}, for $K \in \KKK$ the Cartan matrix $C^K$ at $K$ is a generalized Cartan matrix.

The maps $\rho_i$ are well defined, as for $K \in \KKK$, there exists a unique chamber $K^i$ which is $i$-adjacent to $K$. Since $K$ is then also $i$-adjacent to $K^i$, as the indexing of the root basis is compatible, $\rho_i$ is an involution. Thus $\Cc$ satisfies (C1).

It remains to check (C2). So let $K, L$ be $i$-adjacent with $B^K = \{\alpha_1, \dots, \alpha_r\}$, $B^L = \{\beta_1, \dots, \beta_r\}$. So we find $\rho_i(K) = L$. Let $i,j\in I$. If $i=j$, $c_{ii}^K = 2 = c_{jj}^L$, so assume $i \neq j$.

We find the $i,j$-th entry of $C^K$ to be the number $-c$ such that $\beta_j = c\alpha_i+\alpha_j$. The $i,j$-th entry of $C^{K^i}$ is the number $-d$ defined by
$\alpha_j = d\beta_i + \beta_j$. We obtain $\alpha_j = -d \alpha_i + c\alpha_i+\alpha_j$, and therefore $c = d$ by using the linear independence of $\alpha_i, \alpha_j$. Therefore $\Cc$ satisfies (C2) and is a Cartan graph.

The Cartan graph $\Cc$ is connected: Since $\SSSS$ is a chamber complex, we can find a gallery between two chambers $K$ and $K'$. Let $(K = K_0, \dots, K_m=K')$ be such a gallery. Assume that $K_{j-1}$ and $K_{j}$ are $i_j$-adjacent. Then the map $\sigma^{K_{m-1}}_{i_j} \cdots \sigma^{K_0}_{i_1}$ is in $\Hom(K_0, K_m) = \Hom(K, K')$. 
\end{proof}

\begin{defi}
Given a crystallographic Tits arrangement $(\AAA, T)$ with respect to $R$, we will denote the Cartan graph defined in Proposition \ref{CAtoCS} by $\Cc(\AAA,T,R)$.

Let $K \in \KKK$ and let $\phi_K: V^\ast \to \RR^r$ be the coordinate map of $V^\ast$ with respect to the basis $B^K$. As $R$ is crystallographic, $\phi_K(R) \subset \ZZ^r$. Furthermore let $R^K := \phi_K(R)$ for $K \in \KKK$.
\end{defi}

\begin{lem}
Let $K,K'$ be $i$-adjacent chambers. Then
\begin{eqnarray*}
\phi_{K'} \circ \varphi_{K', K} &=& \phi_K,\\
\phi_{K'} &=& \sigma_i^K \circ \phi_K.
\end{eqnarray*}
\end{lem}
\begin{proof}
This is a straight forward calculation using the definition of $\phi_k$ and $\sigma_i^K$.
\end{proof}

Using induction on the above expressions immediately yields:

\begin{fol}
Let $K_0,K_m$ be arbitrary chambers, $(K_0, \dots, K_m)$ be a gallery such that $K_{j-1}$, $K_j$ are $i_j$ adjacent, $j=1,\ldots,m$. Then
\begin{eqnarray*}
\phi_{K_m} \circ (\varphi_{K_m, K_{m-1}} \cdots \varphi_{K_1,K_0}) &=& \phi_{K_0},\\
\phi_{K_m} &=& (\sigma_{i_{m}}^{K_{m-1}} \cdots \sigma_{i_1}^{K_{0}})\circ \phi_{K_0}.
\end{eqnarray*}
\label{mapconc}
\end{fol}

\begin{prop}
The Cartan graph $\Cc(\AAA,T,R)$ is simply connected.
\end{prop}
\begin{proof}
Let $K \in \KKK$ and $w \in \Hom(K,K)$ such that
$$w = \id^K\sigma_{i_{m}^{K_{m-1}}} \cdots \sigma_{i_1^{K}},$$
so in particular $K = K_0, K_1, \dots, K_m = K$ is a gallery from $K$ to $K$.
By Corollary \ref{mapconc}, $w$ is the identity on $\ZZ^r$.
\end{proof}

\begin{prop}
Let $(\AAA,T)$ be a crystallographic arrangement of rank $r$ with respect to $R$ and $\Cc = \Cc(\AAA,T,R)$.  Then the sets $R^K$ are exactly the real roots of $\Cc$ at $K$, and $\rsC = \rsC(\Cc, (R^K)_{K\in \KKK})$ is a root system of type $\Cc$.
\end{prop}
\begin{proof}
We show that $\rsC$ is a root system of type $\Cc$.
Lemma \ref{purelc} and the crystallographic property imply (R1), and (R2) holds since $R$ is reduced.
To show (R3), assume $K, K'$ are $i$-adjacent. In particular this means $\rho_i(K) = K'$. Now
$$\sigma_i^K(R^K) = \sigma_i^K\phi_K(R) = \phi_{K'}(R) = R^{K'} = R^{\rho_i(K)}$$ 
by Corollary \ref{mapconc}, so (R3) holds.

Let $i \neq j \in I$ and $K \in \KKK$, such that $m_{ij} = |R^K \cap (\NN_0\alpha_i + \NN_0 \alpha_j)|$ is finite.
Assume $B^K = \{\beta_1, \dots, \beta_r\}$, then this is equivalent to $m_{ij} = |R \cap \langle \beta_i, \beta_j \rangle|$, as $\phi_K$ maps $\beta_k$ to $\alpha_k$ for $k \in I$. Take the simplex $F \in \SSSS$, $F \subset \ol{K}$, such that the type of $F$ is $\{i,j\}$, in particular $F$ is a $2$-simplex and $F = \alpha_i^\perp \cap \alpha_j^\perp \cap \ol{K}$. Take $x \in F$ such that $\AAA_x \cap W^K = \{\alpha_i^\perp, \alpha_j^\perp\}$, and consider the arrangement $(\AAA_x, R^x)$. Then $\AAA_x$, $R^x$ has exactly $m_{ij}$ elements, by \ref{Axfin} $(\AAA_x^\pi, R^x)$ is a spherical arrangement, and $\KKK^x$ consists of $2m_{ij}$ chambers. The induced simplicial complex has a unique induced type function by $\{i,j\}$. Therefore $(\rho_i\rho_j)^{m_{ij}}$ corresponds to a unique gallery $(K=K_0, K_1, \dots, K_n)$ of length $2m_{ij}$. Thus we obtain $K_n=K$ and $(\rho_i\rho_j)^{m_{ij}}(K) = K$.

It remains to show that $R^K$ are actually the real roots at $K$. Since $\varphi_{K',K}$ maps roots to roots, we have $(R\re)^K \subset R^K$ by Corollary \ref{mapconc},
so we need to check the other inclusion.
Let $\beta \in R$, and set $\beta_K = \phi_K(\beta)$. Let $K' \in \KKK_0$, such that $\beta \in B^{K'}$, and let $(K = K_0, K_1, \dots, K_m = K')$ be a gallery from $K$ to $K'$ with $K_{i-1}, K_i$ being $j_i$-adjacent.
Let $\alpha \in B^K$ such that 
$$\varphi_{K,K_1}\circ \cdots, \circ\varphi_{K_{m-1},K'}(\beta) = \alpha.$$
Then by Corollary \ref{mapconc}
$$
\beta_K = \phi_K(\beta)= \sigma_{j_1}^{K_1} \circ \cdots \circ \sigma_{j_m}^{K_m}\phi_{K'}(\beta),
$$
where we used the fact that by (C2), $\sigma_i^{K^i}\sigma_i^K = \id_{\ZZ^r}$ if $K^i$ is $i$-adjacent to $K$. Now $\beta \in K'$ yields that $\phi_{K'}(\beta)$ is in the standard basis, which proves $\beta_K \in (R\re)^K$. Hence $R^K = (R\re)^K$, which proves our assumption.
\end{proof}

\begin{rem}
It is easy to see that combinatorially equivalent crystallographic Tits arrangements $(\AAA, T)$ (w.r.t.\ $R$) and $(\AAA',T')$ (w.r.t.\ $R'$) yield equivalent Cartan graphs $\Cc(\AAA,T,R)$ and $\Cc(\AAA',T',R')$. Choosing a different type function of the simplicial complex $\SSSS$ also gives rise to equivalent Cartan graphs, which only differ by a permutation of $I$.
\end{rem}

\subsection{The additive property}

In this section we will discuss the additive property of root systems. Here we will also use some results for subarrangements, which we will prove later in Section \ref{subarr}.

\begin{defi}
Let $(\AAA,T)$ be a Tits arrangement associated to $R$, and fix $K \in \KKK$. Set
\begin{align*}
R^+ &:= R \cap \sum_{\alpha \in B^K} \RR_{\geq0} \alpha, \\
R^- &:= R \cap \sum_{\alpha \in B^K} \RR_{\leq0} \alpha,
\end{align*}
and call $R^+$ the \textit{positive roots (w.r.t.\ $K$)} and $R^-$ the \textit{negative roots (w.r.t.\ $K$)}.
\end{defi}

\begin{lem}
If $(\AAA,T,R)$ is a simplicial arrangement, then $R = R^+\ \dot\cup\ R^-$ for every $K \in \KKK$.
\label{Rposneg}
\end{lem}
\begin{proof}
Let $\alpha \in R$, then $\alpha \in R^+$ or $\alpha \in R^-$ by Lemma \ref{posorneg}.
The sets $R^+$, $R^-$ are disjoint since $\sum_{\alpha \in B^K} \RR_{\geq 0} \alpha \cap \sum_{\alpha \in B^K} \RR_{\leq 0} \alpha = \{0\}$, and $0 \notin R$.
\end{proof}

\begin{defi}
Let $(\AAA,T)$ be a Tits arrangement associated to $R$, and let $K \in \KKK$. We say that $B^K$ \textit{satisfies the additive property}, or shorter that \textit{$B^K$ is additive} if for
all $\alpha \in R^+$ either $\alpha \in B^K$ or $\alpha = \alpha_1 + \alpha_2$ with $\alpha_1, \alpha_2 \in R^+$.

If $B^K$ is additive for all $K \in \KKK$, then $(\AAA,T)$ is said to be \textit{additive (w.r.t.\ $R$)}.
\end{defi}

\begin{rem}
\begin{enumerate}
\item If $(\AAA,T)$ is additive, then $(\AAA,T)$ is also crystallographic. This is just a consequence from the definition.
\item In \cite[Corollary 3.8]{CH11} Cuntz and Heckenberger showed that spherical crystallographic arrangements in dimension 2 are additive. This statement is used in \cite[Theorem 2.10]{CH12} to show that every crystallographic spherical arrangement is additive, thus for spherical arrangements the additive property and the crystallographic property are equivalent. Note that both formulations above actually refer to Weyl groupoids.
\item An example of an affine crystallographic arrangement which is not additive in the above sense is the root system of $\tilde{A}_1$, which is
$$
R(\tilde{A}_1) = \{\alpha_1 + k \gamma, \alpha_2+k\gamma \mid k\in \ZZ\},
$$
where $\{\alpha_1, \alpha_2\}$ is a basis of $(\RR^2)^\ast$ and $\gamma = \alpha_1 + \alpha_2$. There is a chamber $K$ such that $B^K = \{\alpha_1, \alpha_2\}$, but $2\alpha_1 + \alpha_2$ is neither in $B^K$ nor a sum of two positive roots.
\end{enumerate}
\end{rem}

We now give a criterion for a crystallographic arrangement to be additive. The idea for the proof of the following statement is based on \cite[Theorem 2.10]{CH12}, but adapted to our notation.

\begin{prop}
Assume that $(\AAA,T)$ is a crystallographic Tits arrangement with respect to $R$ of rank $r \geq 3$. If $(\AAA,T)$ is $2$-spherical, then it is additive with respect to $R$.
\end{prop}
\begin{proof}
Let $K_0 \in \KKK$ and $\beta \in R^+$ w.r.t.\ $K_0$. Let $K \in \KKK$, such that $\beta \in B^{K}$ and assume $d(K_0,K)= m$. Fix a minimal gallery $\gamma = (K_0, K_1, \dots, K_m=K)$. Let $B^{K} = \{\beta_1, \dots, \beta_r\}$, where $\beta = \beta_1$. If $m = 0$, $\beta$ is already in $B^{K_0}$ and we are done.
So let $m \geq 1$ and assume $\beta \notin B^{K_0}$.

Now assume $K$ and $K_{m-1}$ are $i$-adjacent. As $\beta \in R^+$, we find $D_{\beta^\perp}(K_0) = D_{\beta^\perp}(K)$, hence $i \neq 1$ and a minimal gallery between them can not cross $\beta^\perp$. So let $F := \beta_1^\perp \cap \beta_i^\perp \cap \ol{K}$, then $F$ is an $(n-3)$-simplex by construction.

As $(\AAA,T)$ is $2$-spherical, $F \cap T$ is not empty. Let $x \in F \cap T$ such that $\AAA_x \cap W^K = \{\beta_1^\perp, \beta_i^\perp\}$.
Then $R_x$ is contained in $\langle \beta_1, \beta_i \rangle$, $\KKK_x$ corresponds to the chambers of the star $St(F)$ of $F$. Now $St(F)$ is a gated set by \cite[Proposition 2.27]{CMW}, so let $G \in \KKK_x$ be the unique gate from $K_0$ to $\KKK_x$. Then $B^G \cap R_x = \{\alpha_1, \alpha_2\}$ for some $\alpha_1, \alpha_2 \in \langle \beta_1, \beta_i \rangle$. Let $H_i = \alpha_i^\perp$ for $i=1,2$. By construction of $G$ we find $D_{H_i}(G) = D_{H_i}(K)$ for $i=1,2$ and therefore $K \subset \alpha_1^+ \cap \alpha_2^+$, by Lemma \ref{purelc} the roots $\alpha_1, \alpha_2$ are positive with respect to $K$.

By Corollary \ref{cryst:globtoloc} $R_x$ itself is a crystallographic root system in dimension 2, and $\{\alpha_1, \alpha_2\}$ is a root basis. By construction $G \subset \beta^+$, as $\beta_1 \in R_x$, and again by Lemma \ref{purelc} we obtain that $\beta$ is a positive linear combination of $\alpha_1, \alpha_2$. From \cite[Corollary 3.8]{CH11} it follows that $R_x$ is additive, so $\beta$ is either in $\{\alpha_1, \alpha_2\}$ or sum of two positive roots $\alpha_1', \alpha_2'$ in $R^+ \cap \langle \alpha_1, \alpha_2 \rangle$. In the latter case we are done, as $\alpha_1', \alpha_2'$ are also positive with respect to $K$, since they are positive linear combinations of $\alpha_1, \alpha_2$, which are positive w.r.t.\ $K$.

So it remains to check that $\beta \neq \alpha_1, \alpha_2$. As $G$ is the gate from $K$ to $\KKK_x$, we can assume that there exists an index $0 \leq j \leq m$ such that $K_j = G$ in the above gallery. Assume $\beta = \alpha_1$, then $\beta \in B^G$ and the minimality of the gallery yields $j=m$. But we assumed $K_{m-1}$ and $G$ are $i$-adjacent, which means that $K_{m-1} \in \KKK_x$, a contradiction to the gate property.
So $\beta \neq \alpha_1, \alpha_2$ and we are done.
\end{proof}

\section{The geometric realisation of a connected simply connected Weyl groupoid}
\label{geom.real}

In the previous section, we constructed a Cartan graph from a given crystallographic simplicial arrangement. The aim of this section is to give a canonical crystallographic Tits arrangement associated to a given connected simply connected Cartan graph with real roots.

For this section, assume $\Cc = \Cc(I, A, (\rho_i)_{i\in I}, (C^a)_{a \in A})$ to be a connected simply connected Cartan graph of rank $r$ with real roots $\rsC\re = \rsC(\Cc, ((R\re)^a)_{a\in A})$, and fix some $a \in A$. Furthermore, assume that $\rsC\re$ is a root system of type $\Cc$.
By \cite[Proposition 2.9]{p-CH09a}, this is equivalent to the existence of a root system of type $\Cc$.

\begin{defi}
Let $V = \RR^r$, $I = \{1, \dots, r\}$ and let $B := \{\alpha_i \mid i \in I\}$ be the standard basis of $\ZZ^r$. Assume $\{\beta_i \mid i \in I\}$ is a basis of $V^\ast$. Let $\psi: \ZZ^r \to V$ be the unique $\ZZ$-linear map given by $\alpha_i \mapsto \beta_i$.

Define $R := \psi((R\re)^a)$ and $\AAA := \{r^\perp \mid r \in R\}$. For $b \in A$ with $\Hom(a,b) = \{w\}$, define the map $\psi_b: (R\re)^b \to R$ as $\psi_b = \psi w^{-1}$. In particular, $\psi_a = \psi|_{(R\re)^a}$.
Further let
$$
B^{b} := \psi_b(B)
$$
for all $b \in A$.
Given $B^b$, set 
$$
K^b := \bigcap_{\beta \in B^b} \beta^+,
$$
and let $\KKK = \{K^b \mid b \in A\}$. Note that $B^b$, $K^b$ are defined (and well defined) for all $b\in A$, since $\Cc$ is connected (and simply connected).
The \textit{walls of $K^b$} are the elements of
$$
W^b := \{\alpha^\perp \mid \alpha \in K^b\}.
$$
Let $w \in \Hom(a,b)$ and $i \in I$. We call $K^b \neq K^{b'}$ \textit{$i$-adjacent} if $$\langle\ol{K^b} \cap \ol{K^{b'}} \rangle = \psi_b(\alpha_i)^\perp.$$
We say $K^b$ and $K^{b'}$ are \textit{adjacent} if they are $i$-adjacent for some $i \in I$.
\label{def:CStoSA}
\end{defi}

\begin{fol}
With the above definitions,
$$
\psi_b = \psi_{b'}w
$$
for arbitrary $b,b' \in A$ and $w \in \Hom(b,b')$.
\label{mapconc2}
\end{fol}

\begin{lem}
For $b \in A$, $K^b$ is a simplicial cone, and $H$ does not meet $K^b$ for all $H \in \AAA$.
\label{WG:simpcones}
\end{lem}
\begin{proof}
The set $K^b$ is a simplicial cone by definition, as the sets $B^b$ are bases by construction. Let $w \in \Hom(a,b)$, then $w((R\re)^a) = (R\re)^b$, and we have
$\psi_b(B) = B^b$ as well as $\psi_b((R\re)^b) = R$.
Therefore $R \subset \pm\sum_{\alpha \in B^b} \NN_0 \alpha$. By Lemma \ref{purelc}, we obtain for every $H \in \AAA$ that the vertices of an open simplex $S$ with $K^b = \RR_{>0}S$ which are not contained in $H$ are on the same side of $H$.
\end{proof}

\begin{lem}\label{lemspan}
Let $b \in A$, $H \in W^b$, then $\ol{K^b} \cap H$ spans $H$.
\end{lem}
\begin{proof}
This follows as $K^b$ is a simplicial cone.
\end{proof}

\begin{defi}
For $H \in \AAA$, Lemma \ref{lemspan} yields that every $K^b$ is contained in a unique halfspace associated to $H$. We denote this halfspace by $D_H(K^b)$. For the halfspace not containing $K^b$ we write $-D_H(K^b)$.

We say that $H$ \textit{separates} $K^b$ and $K^{b'}$ for $b,b' \in A$ if $D_H(K^b) = -D_H(K^{b'})$, and set
$$S(K^b, K^{b'}) = \{H \in \AAA \mid D_H(K^b) = -D_H(K^{b'})\}.$$
Furthermore let $T$ be the convex hull of all $K^b$, $b \in A$.
\end{defi}

\begin{lem} (\cite[Lemma B.2]{CMW})
Let $C = \{v_1, \dots, v_r\}$, $C' = \{v_1', \dots, v_r'\}$ be bases of $V$ such that $K^{C} = K^{C'}$. Then, up to permutation, $\alpha_i = \lambda_i\alpha_i'$ for some $\lambda_i \in \RR_{>0}$ for all $1 \leq i \leq r$.

The same holds for two bases $B,B'$ of $V^\ast$ such that $K^B = K^{B'}$.
\label{equalcones}
\end{lem}

We will need the following characterization of walls.

\begin{lem}
Assume $b \in A$ and let $H \subset V$ be a hyperplane. Then $H \in W^b$ if and only if $H \cap K^b = \emptyset$ and $\langle H \cap \ol{K^b} \rangle = H$.
\label{char:wall}
\end{lem}
\begin{proof}
Assume $H \in W^b$ and let $\alpha \in B^b$ such that $\alpha^\perp = H$. Since $K^b \subset \alpha^+$, $K^b \cap H =\emptyset$. By definition of $K^b$ the set $\ol{K^b} \cap H$ is not empty. Let $S$ be a closed simplex such that $\ol{K^b} = \RR_{>0}S \cup \{0\}$, by Remark \ref{rem:scs} it follows that there exists a maximal face $F$ of $S$ contained in $H$. But $F$ has an $n-2$-dimensional affine space as its affine span, therefore its linear span is a hyperplane. Furthermore $F \subset H \cap \ol{K^b}$, hence $\langle H \cap \ol{K^b} \rangle = H$.

Now assume $H \cap K^b = \emptyset$ and $\langle H \cap \ol{K^b} \rangle = H$ both hold. The set $K^b$ is a simplicial cone, from Remark \ref{rem:scs} we obtain that there exist elements $\beta_1, \dots, \beta_r \in V^\ast$ such that 
$$\ol{K^b} = \bigcap_{i=1}^r \beta_i^+.$$
By using \ref{equalcones} we can assume $B^b = \{\beta_1, \dots, \beta_r\}$. Let $S$ be as above, then we find a maximal face $F$ of $S$ such that $F \subset H$, but every face of $S$ is contained in a unique hyperplane $\beta_i^\perp$, which proves our claim.
\end{proof}

\begin{lem}
The map $A \to \KKK$, $b \mapsto K^b$, is a bijection.
\label{bKbbij}
\end{lem}
\begin{proof}
As $\mathcal{R}\re$ is a root system of type $\Cc$, it follows from \cite[Lemma 8, (iii)]{HY08}, that $\Cc$ satisfies (C3), which implies the statement.
\end{proof}

\begin{prop}
Let $b,b' \in A$, $i \in I$, and let $S := \bigcap_{i \neq j \in I} (\psi_b(\alpha_j)^+ \cap \psi_{b'}(\alpha_j)^+)$. Then the following are equivalent:
\begin{enumerate}[label=\roman*)]
	\item $K^b$ and $K^{b'}$ are $i$-adjacent,
	\item $\rho_i(b) = b'$,
	\item $\forall K \in \KKK: (K \subset S \Longleftrightarrow K \in \{K^b, K^{b'}\})$,
	\item $S(K^b, K^{b'}) = \{\psi_b(\alpha_i)^\perp\}$.
\end{enumerate}
\label{char:adjacency}
\end{prop}
\begin{proof}
ii) $\implies$ iii),iv): Assume $\rho_i(b) = b'$, then $B^{b'} = \{\psi_b(\alpha_j - c_{ij}^b \alpha_i) \mid j \in I\}$ and by definition we have $\psi_{b'}(\alpha_j) = \psi_b(\alpha_j-c_{ij}^b\alpha_i)$.
Therefore $K^b \subset \psi_b(\alpha_j)^+$ for all $j \in I$, and hence also $K^b \subset \psi_{b'}(\alpha_j)^+$ for all $i \neq j \in I$. The analogue statement holds for $K^{b'}$, so $K^b, K^{b'} \subset S$. Now assume $K \in \KKK$, $K \subset S$, and consider the case $K \subset \psi_b(\alpha_i)^+$. Then we obtain $K \subset \bigcap_{i \in I}\psi_b(\alpha_i)^+$ and thus $K \subset K^b$. By Lemma \ref{WG:simpcones} this already implies $K = K^b$. The case $K \subset \psi_b(\alpha_i)^-$ yields in the same way $K = K^{b'}$. Thus iii) holds.
This also implies $\psi_b(\alpha_i)^\perp \in S(K^b, K^{b'})$. Hence we get
\begin{align*}S &=
(S \cap \psi_b(\alpha_i)^+)\ \dot\cup\ (S \cap \psi_b(\alpha_i)^-) \ \dot\cup\ (S \cap \psi_b(\alpha_i)^\perp)\\
&= K^b\ \dot\cup\ K^{b'}\ \dot\cup (S \cap \psi_b(\alpha_i)^\perp,
\end{align*}
so assume $H \in S(K^b, K^{b'})$. Then $H$ must meet $S$, but cannot meet $K^b$ or $K^{b'}$ by Lemma \ref{WG:simpcones}. The intersection $H \cap S$ is open in $\psi_b(\alpha_i)^\perp$. Hence the two hyperplanes must coincide. This shows iv).

iv) $\implies$ iii), i): Let $S(K^b, K^{b'}) = \{\psi_b(\alpha_i)^\perp\}$. By definition $\psi_b(\alpha_i)^\perp$ is a wall of $K^b$. Assume $\psi_b(\alpha_i)^\perp$ is not a wall of $K^{b'}$, then $K^b \subset \bigcap_{i \in I} \psi_{b'}(\alpha_i)^+$ and $K^b = K^{b'}$ by Lemma \ref{WG:simpcones}, but then $S(K^b, K^{b'}) = \emptyset$, a contradiction, hence $\psi_b(\alpha_i)^\perp \in W^{b'}$ holds. For $j \neq i$ we find $D_{\psi_b(\alpha_j)^\perp}(K^b) = D_{\psi_b(\alpha_j)^\perp}(K^{b'})$, and the same holds for $\psi_{b'}(\alpha_j)^\perp$. Therefore $S$ contains both $K^b$ and $K^{b'}$. Assume $K \subset S$, then $K$ is on either side of $\psi_b(\alpha_i)^\perp$. In case $D_{\psi_b(\alpha_i)^\perp}(K) = D_{\psi_b(\alpha_i)^\perp}(K^b)$, $K = K^b$ holds, so assume $D_{\psi_b(\alpha_i)^\perp}(K) = D_{\psi_b(\alpha_i)^\perp}(K^{b'})$. As $\psi_{b}(\alpha_i)^\perp$ is a wall of $K^{b'}$ different from $\psi_{b'}(\alpha_j)^\perp$ for $j \neq i$, we already have $\psi_{b}(\alpha_i)^\perp = \psi_{b'}(\alpha_i)^\perp$. We obtain $K \subset \bigcap_{i \in I} \psi_{b'}(\alpha_i)^+$, and therefore $K=K^{b'}$, which shows iii).
Furthermore we see that $S \cap \psi_b(\alpha_i)^\perp$ is not empty, as $S$ is a convex set containing points in $\psi_b(\alpha_i)^+$ and in $\psi_b(\alpha_i)^-$. In particular we showed
\begin{align*}S &=
(S \cap \psi_b(\alpha_i)^+)\ \dot\cup\ (S \cap \psi_b(\alpha_i)^-) \ \dot\cup\ (S \cap \psi_b(\alpha_i)^\perp)\\
&= K^b\ \dot\cup\ K^{b'}\ \dot\cup (S \cap \psi_b(\alpha_i)^\perp,
\end{align*}
 and $(S \cap \psi_b(\alpha_i)^\perp \subset \ol{K^b} \cap \ol{K^{b'}}$. Consider open balls $U \subset K^b$, $U' \subset K^{b'}$. The convex hull of $U$ and $U'$ is again open, and therefore intersects $\psi_b(\alpha_i)^\perp$ in subset $U''$, which is open in $\psi_b(\alpha_i)^\perp$. Hence $U''$ spans $\psi_b(\alpha_i)^\perp$. Now $U''$ is contained in $\ol{K^b}$ as well as $\ol{K^{b'}}$, so $K^b$ and $K^{b'}$ are $i$-adjacent, which shows i).

i) $\implies$ iv): Let $K^b$ and $K^{b'}$ be $i$-adjacent, so $\langle \ol{K^b} \cap \ol{K^{b'}} \rangle = \psi_b(\alpha_i)^\perp$. Let $H = \psi_b(\alpha_i)^\perp$. Assume $H' \in \AAA$ separates $K^b$ and $K^{b'}$. Then $(\ol{K^b} \cap H) \cap (\ol{K^{b'}} \cap H)$ will be contained in $H'$. If this intersection spans a hyperplane, then $H = H'$. Therefore iv) holds.

iii) $\implies$ ii): We have the equality $S \cap \psi_b(\alpha_i)^+ = K^b$ by definition. The intersection $S \cap \psi_b(\alpha_i)^-$ must contain $K^{b'}$. Furthermore we find that $\psi_b(\alpha_i)^\perp$ and $\psi_{b'}(\alpha_i)^\perp$ separate $K^b$ and $K^{b'}$, because otherwise $K^b \subset K^{b'}$. Since $S$ is convex, we actually find an open subset $U' \subset S$ which is in $\psi_b(\alpha_i)^\perp \cap \ol{K^b}$. Since $S$ is open, this is contained in an open subset $U \subset S$, such that $U \cap \psi_b(\alpha_i)^\perp = U'$.
Assume $\rho_i(b) = b''$, then ii) $\implies$ iv) $\implies$ i) yields that
$$S' = \bigcap_{i \neq j \in I} (\psi_b(\alpha_j)^+ \cap \psi_{b''}(\alpha_j)^+)$$ contains exactly the chambers $K^b$ and $K^{b''}$ and that $K^b$, $K^{b''}$ are $i$-adjacent. Furthermore we obtain
$$S' = K^b\ \dot\cup\ K^{b''}\ \dot\cup\ (S' \cap \psi_b(\alpha)^\perp).$$

By construction $S'$ contains $U'$, therefore it also contains an open set $U''$ such that $U'' \cap \psi_b(\alpha_i)^\perp = U'$. But then $U \cap U''$ is open, contained in $S$, and meets $K^{b''}$. So $K^{b'} = K^{b''}$ and hence $b' = b''$ by Lemma \ref{bKbbij}. This shows ii) and finishes the proof.
\end{proof}

\begin{lem}
For every $H \in \AAA$, there exists $b \in A$ such that $H \in W^b$.\label{Hiswall}
\end{lem}
\begin{proof}
As $H \in \AAA$, we find $\alpha \in R$ such that $H = \alpha^\perp$, thus there exists $b \in A$ and $i \in I$ such that $\alpha = \phi_b(\alpha_i)$, and consequently $H \in W^b$.
\end{proof}

\begin{lem}
We have $|S(K^b, K^{b'})| = 0$ if and only if $b=b'$, and $|S(K^b, K^{b'})| = 1$ if and only if $K^b$ and $K^{b'}$ are adjacent.
\label{char:dist:anch}
\end{lem}
\begin{proof}
Assume $|S(K^b, K^{b'})| = 0$ holds, and let $B^b := \{\beta_1, \dots, \beta_r\}$. Then $K^{b'} \subseteq \bigcap_{i=1}^r\beta_i^+$. Assuming that the inclusion is proper provides a contradiction to the fact that no hyperplane in $W^{b'}$ meets $K^b$ in Lemma \ref{WG:simpcones}.

In the case $b=b'$, the statement in Lemma \ref{bKbbij} yields $|S(K^b, K^{b'})| = 0$, which shows the first equivalence.

For the second statement assume $S(K^b, K^{b'}) = \{H\}$, and assume $H \notin W^b$. Then $K^{b'} \subset D_{H'}(K^b)$ for all $H' \in W^b$, and therefore $K^{b'} \subset K^b$ holds. As before we obtain $b=b'$ and $|S(K^b, K^{b'})| = 0$ in contradiction to our assumptions. 
Therefore $H$ is a wall of $K^b$ as well as $K^{b'}$, hence $H = \psi_b(\alpha_i)^\perp$ for some $i \in I$. The statement follows from Proposition \ref{char:adjacency}, iv) $\implies$ i). 

If $K^b$, $K^{b'}$ are $i$-adjacent, Proposition \ref{char:adjacency}, i) $\implies$ iv), yields
$|S(K^b, K^{b'})| = 1$.
\end{proof}

\begin{remdef}
We define the distance function on $\KKK$ in the following way, which is the same definition as for chamber complexes. Take a minimal gallery $K^{b_0}, \dots, K^{b_m}$ between chambers $K^{b_0}$ and $K^{b_m}$, where $K^{b_i}$ and $K^{b_{i+1}}$ are adjacent (this exists since $\Cc$ is connected), and set
$$d(K^{b_0}, K^{b_m}) := m.$$
This is a well defined metric on the set $\KKK$, thus $(\KKK,d)$ is a metric space. The fact that $d$ is a metric can be checked in the same way as if $\KKK$ was constructed from a simplicial arrangement.
\end{remdef}

\begin{prop}
Assume $K^b \in \KKK$ and $x \in \ol{K}$ such that $\AAA_x = \{H \in \AAA \mid x \in H\}$ is finite. Let $R_x = \{\alpha \in R \mid \alpha^\perp \in \AAA_x\}$, $W = \langle R_x \rangle$, $V_x = V/W^\perp$. Construct $I_x = \{i \in I \mid \psi_b(\alpha_i) \in R_x\}$, and set $\Pi_x = \langle \rho_i \mid i \in I_x\rangle$. Further define $A_x := \Pi_x(b)$ and $C^{b'}_x = (c^{b'}_{i,j})_{i,j \in I_x}$. Set
$$
\mathcal{C}_x := \mathcal{C}(I_x, A_x, (\rho_i)_{i \in I_x}, (C^{b'}_x)_{b' \in A_x}),
$$
then $\mathcal{C}_x$ is a
connected simply connected
Cartan graph which admits a finite root system
with real roots at $c$ being the set $\psi_c^{-1}(R_x)$ for $c \in A_x$. Here $R_x$ is identified with a subset of $(V_x)^\ast$ via $\alpha(v + W^\perp) := \alpha(v)$.
\label{char:res}
\end{prop}
\begin{proof}
First notice that $\mathcal{C}_x$ is indeed a connected and simply connected Cartan graph by Lemma \ref{residues} as it is an $I_x$-residue of $\Cc$.

Denote by $(R\re_x)^c$ for $c \in A_x$ the real roots at $c$ given by the Cartan graph $\Cc_x$. By taking the standard basis $\{\alpha_1, \dots, \alpha_r\}$, we can consider $\ZZ^{I_x}$ as the $\ZZ$-span of $\{\alpha_i \mid i \in I_x\}$ in $\ZZ^I$.
Comparing the construction of
$$(R\re_x)^c = \{\id^c\sigma_{i_1}\cdots \sigma_{i_k}(\alpha_j) \mid k \in \NN_0, i_1, \dots, i_k, j \in I_x\}$$
 and
$$(R\re)^c = \{\id^c\sigma_{i_1}\cdots \sigma_{i_k}(\alpha_j) \mid k \in \NN_0, i_1, \dots, i_k, j \in I\},$$ it follows immediately that $(R\re_x)^c \subseteq \psi_c^{-1}(R_x)$ for $c \in A_x$. In particular $(R\re_x)^c$ is
a finite root system,
and for $c \in A_x$ the real roots $(R\re_x)^c$ at $c$ form a spherical simplicial arrangement of rank $|I_x|$ via the map $\psi_c$ in the space $V_x$. Call this arrangement $\AAA'$, the corresponding root system $R'$ and the chambers $\KKK'$. Also denote the canonical projections with $\pi_x: V \to V/W^\perp$ and $\pi_x^\perp: V^\ast \to (V_x)^\ast$, $\pi_x^\perp(\alpha)(y+W^\perp) =  \alpha(y)$.
Let $c \in A$, $d \in A_x$, and let $(K')^c$ be the chamber associated to $c$ in $\Cc_x$, and $(B')^c$ the respective root basis in $(V_x)^\ast$.

We show that for all $c \in A_x$ we have $\pi_x(K^c) = (K')^c$ and $\pi_x^\perp(B^c \cap R_x) = (B')^c$. For the simplicial arrangement associated to $\Cc_x$ we need some notation as in Definition \ref{def:CStoSA}. Let $\psi': \ZZ^{I^x} \to (V_x)^\ast$, mapping $\alpha_i$ to $\beta_i$, i.e.\ $\psi' = \psi_b|_{\ZZ^{I_x}}$. The associated root system is then given by $R' = \psi'((R_x\re)^b)$. For $c \in A_x$ and $\Hom(c,b) = \{w\}$ we can then define $\psi'_c = \psi' w$ and get $\psi_c((R_x\re)^c) = R'$.

We find by definition of $I_x$ that $B^c \cap R_x = \psi_c(\{\alpha_i\mid i \in I_x\})$ and $(B')^c = \psi'_c(\{\alpha_i\mid i \in I_x\})$ via the embedding of $\alpha_i$, $i \in I_x$ into $\ZZ^I$. By definition we can write $\psi'_c = \psi' w$, $\psi_c = \psi_b w$ for $\Hom(c,b) = \{w\}$. As noted above $\psi' = \psi_b|_{\ZZ^{I_x}}$, so $\psi'_c = \psi_b|_{\ZZ^{I_x}} w$, which yields the equality $\pi_x^\perp(B^c \cap R_x) = (B')^c$. Given this, we immediately obtain $\pi_x(K^c) = (K')^c$ by considering how $K^c$, $(K')^c$ are defined.

Now assume $\alpha \in R_x$. If it is not in $\psi_b(R_x\re)$, it meets a chamber in $\KKK'$, since the elements in $\KKK'$ are the connected components of $V_x \setminus \bigcup_{H \in \AAA_x}H$, which is impossible by Lemma \ref{WG:simpcones}.
We conclude $R_x = R'$, as required.
\end{proof}

\begin{rem}
It is an easy observation (the proof is similar to the proof of \cite[Lemma 2.7]{CMW}) that the separating hyperplanes for two chambers $K^b, K^{b'}$ with $x \in \ol{K^b} \cap \ol{K^{b'}}$ are all contained in $A_x$. In combination with the next lemma, this yields that the Cartan graph $\Cc_x$ is independent of the choice of $K^b$. In other words, in this case $b' \in \Pi^x(b)$.
\label{WG:resioc}
\end{rem}

The next lemma yields a characterization of the distance $d$, which we already established for simplicial arrangements.

\begin{lem}
If $b,b' \in A$, then $d(K^b, K^{b'}) = |S(K^b, K^{b'})|$.
\label{char:dist}
\end{lem}
\begin{proof}
Let $d := d(K^b, K^{b'})$, and set $m = |S(K^b, K^{b'})|$ in case this is finite. For $d = 0,1$ the statement follows from Lemma \ref{char:dist:anch}. We prove $|S(K^b, K^{b'})| \leq d$ by induction on $d$, so assume $d \geq 2$ and for $c,c' \in A$ with $d(K^c, K^{c'}) < d$ we know $d(K^c, K^{c'}) = |S(K^c, K^{c'})|$. Let $K^b = K^{b_0}, K^{b_1}, \dots, K^{b_d} = K^{b'}$ be a minimal gallery from $K^b$ to $K^{b'}$. We find $K^{b_1}, \dots, K^{b_d}$ to be a minimal gallery from $K^{b_1}$ to $K^{b'}$, and hence $d(K^{b_1}, K^{b'}) = d-1 = |S(K^{b_1}, K^{b'})|$ by induction, and $d(K^b, K^{b_1}) = 1 = |S(K^b, K^{b_1})|$.

So we can assume $S(K^b, K^{b_1}) = \{ H_1\}$, $S(K^{b_1}, K^{b'}) = \{H_2, \dots, H_d\}$.

Assume $H$ separates $K^b$ and $K^{b'}$, and let $1 \leq j \leq d$ be the first index, such that $D_H(K^{b_{j-1}}) = D_H(K^b)$, $D_H(K^{b_{j}}) = D_H(K^{b'})$. This index exists, otherwise we find $D_H(K^b) = D_H(K^{b'})$. If $j = 1$, we find $H = H_1$, else we find $H \in S(K^{b_1}, K^{b'})$, and we can conclude
$$S(K^{b}, K^{b'}) \subset S(K^{b}, K^{b_1}) \cup S(K^{b_1}, K^{b'})
$$
and thus $m \leq d$.

To show equality we show that there exists a gallery of length $m$ connecting $K^b$ and $K^{b'}$. As we now know that $S(K^{b}, K^{b'})$ is actually finite, let $S(K^{b}, K^{b'}) = \{H'_1, \dots, H'_m\}$. There exists a hyperplane in $S(K^{b}, K^{b'})$ which is a wall of $K^b$, otherwise we find for every wall of $K^b$, that $K^{b'}$ is on the same side, which yields $K^b = K^{b'}$ and $b=b'$, in contradiction to $d(K^b, K^{b'}) \geq 2$.

So assume $H'_1$ is a wall of $K^b$, then we find $H'_1 = \psi_b(\alpha_i)^\perp$ for some $i \in I$. So let $b_1 \in A$ such that $\rho_i(b) = b_1$, then $K^{b_1}$ is $i$-adjacent to $K^b$ by Proposition \ref{char:adjacency}. We obtain $S(K^b, K^{b_1}) = \{H_1'\}$. As $H_1'$ is the only hyperplane separating $K^b$ and $K^{b_1}$ by Lemma \ref{char:dist:anch}, this implies $D_{H_i'}(K^b) = D_{H_i'}(K^{b_1}) = -D_{H_i'}(K^{b'})$ for $i=2, \dots, m$.

Furthermore note that every $H_i'$ for $2 \leq i \leq m$ separates $K^b$ and $K^{b'}$, and therefore separates $K^{b_1}$ and $K^{b'}$ as well. On the other hand if $H$ separates $K^{b_1}$ and $K^{b'}$, it also separates $K^b$ and $K^{b'}$. We obtain $S(K^{b_1}, K^{b'}) = \{H_2', \dots, H_m'\}$, and by induction we find a gallery of length $m$ connecting $K^b$ and $K^{b'}$. Hence $d \leq m$, which concludes the proof.
\end{proof}

\begin{fol}
Assume $b,b' \in A$ and $K^{b} = K^{b_0}, \dots, K^{b_m}= K^{b'}$ is a minimal gallery from $K^b$ to $K^{b'}$. Then this gallery crosses no hyperplane in $\AAA$ more than once.
\end{fol}

Another consequence of Lemma \ref{char:dist} is the following.

\begin{lem}
Let $x \in K^b$, $y \in K^{b'}$. For every point $z \in [x,y]$ there exists a neighbourhood $U_z$ such that $\sec(U_z) = \{H \in \AAA \mid H \cap U_z \neq \emptyset\}$ is finite.
\label{segmfin}
\end{lem}
\begin{proof}
Take an open neighbourhood $U_x$ of $x$ in $K^b$, and an open neighbourhood $U_y$ of $y$ in $K^{b'}$. Take the union $U := \bigcup_{x' \in U_x, y' \in U_y} \sigma(x',y')$. Then $U$ contains $[x,y]$ and every hyperplane that meets $U$ separates $K^b$ and $K^{b'}$. Hence the set $\sec(U)$ is finite by Lemma \ref{char:dist}. We find $\varepsilon, \delta \in \RR_{>0}$ such that the open balls $B_\varepsilon(x)$, $B_\delta(y)$ satisfy $B_\varepsilon(x) \subset U_x$, $B_\delta(y) \subset U_y$. Assume $\varepsilon \geq \delta$, and let $z' \in B_\delta(z)$. Write $z' = z + v$ for $v \in V$. Then $x +v \in B_\delta(x)$, $y + v \in B_\delta(y)$ and
$$z' = z + v \in [x,y] +v = [x',y'] \subset U.$$
Choosing $U_z$ as $B_\delta(z)$ therefore satisfies $|\sec(U_z)| < \infty$.
\end{proof}

Remember that $T$ is defined as the convex hull of $K^b$, $b \in A$. In the following, we show that every point of $T$ is on an interval between two points in the interior of two chambers, which implies that $T$ is indeed convex.

\begin{lem}
Let $b,b' \in A$ with $d(K^b, K^{b'}) = m$, and let
$\Gamma(b,b')
$ the set of minimal galleries from $K^b$ to $K^{b'}$. Let $x \in K^b$, $y \in K^{b'}$. Then
$$
[x,y] \subset \bigcup_{K^c \in \gamma \in \Gamma(b,b')} \ol{K^c}.
$$
\label{mingalcontsegm}
\end{lem}
\begin{proof}
The interval $[x,y]$ only meets the finite set of hyperplanes $S(K^b, K^{b'})$ by Lemma \ref{segmfin}. As $x$, $y$ are not contained in any hyperplane in $\AAA$, the interval $[x,y]$ is not contained in a hyperplane as well.
Let $x_1, \dots, x_k \in [x,y]$ be the points such that $\sec( (x_i, x_{i+1})) = \emptyset$, $\AAA_{x_j} \neq \emptyset$ for all $1 \leq i \leq k-1$, $1 \leq j \leq k$, and $\sec([x,y]) = \bigcup_{i=1}^k \AAA_{x_i}$. These points exist by Lemma \ref{char:dist}.

We show the statement by induction on $k$. Assume $k = 0$, then $b=b'$ and $K^b = K^{b'}$ by Lemma \ref{bKbbij}. As $x,y \in K^b$ and $K^b$ is convex, $[x,y] \subset K^b$, as required.
If $k=1$ then $(x,x_1) \subset K^b$, $(x_1, y) \subset K^{b'}$. Let $K^b = K^{b_0}$, $\dots$, $K^{b_m} = K^{b'}$ be a minimal gallery from $K^b$ to $K^{b'}$. By definition $x_1 \in \ol{K^{b}}$, and $x_1 \in \ol{K^{b'}}$. As $x_1 \in \ol{K^b} \cap \psi_b(\alpha_i^\perp)$, it is also contained in $\ol{K^{b_1}}$, and inductively we obtain $x_1 \in \ol{K^{b_i}}$ for $1 \leq i \leq m$. So our claim holds for $k=1$.

So let $k \geq 2$. We show that every open segment $(x_j, x_{j+1})$ is contained in some chamber $K^{c_j}$ for $1 \leq j <k$. It is enough to show that $(x_1, x_2)$ is contained in a chamber, then the statement follows inductively by substituting $x$ with a point on $(x_1,x_2)$. As $x_1 \in \ol{K^b}$, let $J \subset I$ such that $j \in J$ if and only if $x_1 \in \psi_b(\alpha_j)^\perp$. 

Define $R_{x_1} := \{\alpha \in R \mid x_1 \in \alpha^\perp\}$, $W = \langle R_{x_1} \rangle$, let $\pi_{x_1}^\ast: V^{\ast} \to (V/W^\perp)^\ast$, $\pi_{x_1}^\ast(\alpha)(v+W^\perp) = \alpha(v)$, $\pi_{x_1}: V \mapsto V/W^\perp, v \mapsto v + W^\perp$.
As $\AAA_{x_1}$ is finite, we can apply Proposition \ref{char:res} to find a spherical Cartan graph $\Cc_{x_1}$, together with a set of chambers in one to one correspondence to the objects $A_{x_1}$. In particular, as $\Cc_{x_1}$ is spherical, there exists an object $c_1 \in A_{b,J}$, such that the chamber $K^{c_1}$ has maximal distance to $K^b$ in the spherical arrangement associated to $\Cc_{x_1}$. Let $B^b \cap R_{x_1} = \{\beta_1, \dots, \beta_l\}$, then $W^{c_1} = \{-\beta_1, \dots, -\beta_l\}$, hence $(x_1, x_2) \subset K^{c_1}$ follows.
We can conclude that $(x_j, x_{j+1})$ is contained in $K^{c_j}$ for $1 \leq j < k$.

Let $z \in (x_j, x_{j+1})$ for some $1 \leq j <k$. By counting separating hyperplanes we obtain $d(K^b, K^{c_j}) + d(K^{c_j}, K^{b'}) = d(K^b, K^{b'})$, hence there is a minimal gallery $K^b = K^{b_0}, \dots, K^{b_\lambda} = K^{c_j}, \dots, K^{b'} = K^{b_m}$.
By induction we obtain
$$
[x,y] = [x,z] \cup [z,y] \subset \bigcup_{K^{c'} \in \gamma \in \Gamma(b,c)} \ol{K^{c'}} \cup \bigcup_{K^{c'} \in \gamma \in \Gamma(c,b')} \ol{K^{c'}} \subset \bigcup_{K^{c'} \in \gamma \in \Gamma(b,b')} \ol{K^{c'}},
$$
since every minimal gallery containing $K^c$ yields minimal galleries from $K^b$ to $K^c$ as well as from $K^c$ to $K^{b'}$.
\end{proof}

\begin{lem}
For $b,b' \in A$ the set $\{K^c \mid K^c \in \gamma \in \Gamma(b,b')\}$ is finite.
\label{mingalfin}
\end{lem}
\begin{proof}
There exist only $|I|$ chambers adjacent to $K^b$, inductively there exist only finitely many chambers $K^c$ such that $d(K^b,K^c) \leq d(K^b, K^{b'})$, and $\{K^c \mid K^c \in \gamma \in \Gamma(b,b')\}$ is contained in this set.
\end{proof}

\begin{fol}\label{mgcs2}
Let $x \in \ol{K^b}$, $y \in \ol{K^{b'}}$ for $b,b' \in A$. Then 
$$[x,y] \subset \bigcup_{K^c \in \gamma \in \Gamma(b,b')} \ol{K^c}.$$
\end{fol}
\begin{proof}
Let $p: [0,1] \mapsto [x,y]$ be a continuous parametrization with $p(0) = x$, $p(1) =y$.
Let $x' \in K^b$, $y' \in K^{b'}$ and parametrize the segments $[x,x']$, $[y,y']$ continuously with $p_x: [0,1] \to [x,x']$, $p_y: [0,1] \to [y,y']$, such that $p_x(0) = x$, $p_x(1) = x'$, $p_y(0) = y$ and $p_y(1) = y'$ .  Then by Lemma \ref{mingalcontsegm} we have
$$[p_x(\varepsilon),p_y(\varepsilon)] \subset \bigcup_{K^c \in \gamma \in \Gamma(b,b')} \ol{K^c}
$$
for all $0 < \varepsilon \leq 1$. As $\bigcup_{K^c \in \gamma \in \Gamma(b,b')} \ol{K^c}$ is a finite union of closed chambers by Lemma \ref{mingalfin}, it is closed again. Since $p_x$, $p_y$ are continuous,
$[x,y] = [p_x(0), p_y(0)] \subset \bigcup_{K^c \in \gamma \in \Gamma(b,b')} \ol{K^c}$.
\end{proof}

\begin{lem}\label{char:T}
Let $T_0 := \bigcup_{b \in A} K^b$. The cone $T$ satisfies
$$
T = \bigcup_{x,y \in T_0} [x,y].
$$
\end{lem}
\begin{proof}
Let $T' := \bigcup_{x,y \in T_0} [x,y]$, then the inclusion $T' \subset T$ holds since $T$ is convex. We show that $T'$ is convex.
Let $z, z' \in T'$, then we find $b,b',c, c' \in A$ with $x \in K^b$, $x' \in K^{b'}$, $y \in K^{c}$, $y' \in K^{c'}$, such that $z \in [x,y]$, $z' \in [x',y']$. It follows from Lemma \ref{mingalcontsegm} that there exist chambers $K^{d}$, $K^{d'}$ with $z \in \ol{K^d}$, $z' \in \ol{K^{d'}}$.
By Corollary \ref{mgcs2} we obtain
$$[z,z'] \subset \bigcup_{K^c \in \gamma \in \Gamma(d,d')} \ol{K^c}.$$
Let $z'' \in [z,z']$, then $\AAA_{z''}$ is finite, and there exists a chamber $K^{d''}$ such that $z'' \in \ol{K^{d''}}$. By Proposition \ref{char:res} we obtain that there exists an object $d^\ast$ of maximal distance to $d''$ in the spherical Weyl groupoid induced at $z''$. Therefore $K^{d^\ast}$ has maximal distance to $K^{d''}$ in the respective spherical arrangement, and $z''$ is on a segment between a point in $K^{d''}$ and $K^{d^\ast}$.
\end{proof}

\begin{prop}
The pair $(\AAA, T)$ is a crystallographic Tits arrangement with respect to $R$.
\end{prop}
\begin{proof}
By Lemma \ref{char:T} every point $z \in T$ is on an interval $[x,y]$, $x \in K^b$, $y \in K^{b'}$, and by Lemma \ref{segmfin} we find a neighbourhood $U_z$ such that $\sec(U_z)$ is finite. Therefore $\AAA$ is locally finite in $T$.

Let $K$ be a connected component of $T \setminus (\bigcup_{H \in \AAA} H)$, and let $x \in K$. As a direct result of Lemma \ref{mingalcontsegm} $x$ is contained in $\ol{K^b}$ for some $b \in A$. By definition of $K$ it follows that $K \subset K^b$, and by Lemma \ref{WG:simpcones} we get equality.

Furthermore $\AAA$ is thin, as by definition every wall of $K^b$, $b\in A$, is in $\AAA$. From Lemma \ref{Hiswall} it follows that every $H \in \AAA$ is a wall of some chamber, so it meets $T$.

It follows also by definition that $\AAA = \{\alpha^\perp \mid \alpha \in R\}$, so $R$ is a root system for $\AAA$. Furthermore $R$ is crystallographic since $R = \psi_b(R^b)$, so every root is a non-negative or non-positive integral linear combination of $B^b$ for all $b \in A$.

Finally $R$ is reduced since the roots $(R\re)^a$ satisfy property (R2).
\end{proof}

\begin{rem}
Choosing a different object $a' \in A$ to begin with yields a combinatorially equivalent crystallographic arrangement, as those differ exactly by an element in $\GL_r(\ZZ)$.
Since equivalent Cartan graphs have the same sets of real roots, those also yield combinatorially equivalent crystallographic arrangements.
\end{rem}

\begin{defi}
Given a connected simply connected Cartan graph $\Cc$ permitting a root system of type $\Cc$, we call the crystallographic arrangement $(\AAA,T)$ as constructed above \textit{the geometric realization of $\Cc$ (with respect to $a$)}.
\end{defi}

\begin{fol}\label{cor:cor}
There exists a one-to-one correspondence between connected, simply connected Cartan graphs permitting a root system and crystallographic Tits arrangements with reduced root system.

Under this correspondence, equivalent Cartan graphs correspond to combinatorially equivalent Tits arrangements and vice versa, giving rise to a one-to-one correspondence  between the respective equivalence classes.
\end{fol}

\section{Parabolic subarrangements and restrictions of crystallographic arrangements}
\label{subarr}

In this section we consider substructures of crystallographic Tits arrangements and show that the crystallographic property is inherited by these substructures in most cases.

\subsection{Subarrangements}

\begin{defi}
Assume $(\AAA,T)$ is a Tits arrangement associated to $R$, let $x \in \ol{T}$. Set
\begin{align*}
W_x &:= \supp_\AAA(\{x\}),&
V_x &:= V/W_x,\\
\pi &:= \pi_x: V \mapsto V_x, v \to v + W_x,&
T_x &:= \pi(T),\\
\AAA_x^\pi &:= \pi(\AAA_x),&
\KKK_x &:= \{K \in \KKK\ |\ x \in \overline{K}\},\\
R_x &:= \{\alpha \in R \mid \alpha^\perp \in \AAA_x\},&
B^K_x &:= B^K \cap R_x.
\end{align*}
\end{defi}

The following are the main results about parabolic subarrangements from \cite{CMW}, they describe how to understand a parabolic subarrangement as a Tits arrangement.

\begin{prop}[{\cite[Proposition 4.4]{CMW}}]
Let $(\AAA, T)$ be a Tits arrangement associated to $R$, $x \in T$. Then $(\AAA_x^\pi, T_x)$ is a Tits arrangement associated to $R_x$.
\end{prop}

\begin{fol}[{\cite[Corollary 4.6]{CMW}}]
Assume $(\AAA,T)$ is a Tits arrangement associated to a root system $R$ of rank $r \geq 2$. Let $x \in \ol{T}$. Then $\AAA_x$ and $R_x$ are finite if and only if $x \in T$.
In particular, a Tits arrangement is finite if and only if it is spherical.\label{Axfin}
\end{fol}

We recall the following statement from \cite{CMW}.

\begin{lem}[{\cite[Lemma 4.8]{CMW}}]
Let $(\AAA,T)$ be a Tits arrangement associated to a root system $R$. Let $x \in \ol{T}$ with $R_x \neq \emptyset$. Let $K,L \in \KKK_x$ be $\alpha_1$-adjacent, and $B^K = \{\alpha_1, \dots, \alpha_r\}$, $B^L = \{\beta_1, \dots, \beta_r\}$ indexed compatibly with $B^K$. Then $B^{K}_x \rightarrow B^{L}_x$, $\alpha_i \mapsto \beta_i$ is a bijection.
\label{restmap}
\end{lem}

This can be applied to crystallographic arrangements.

\begin{fol}
Let $(\AAA,T)$ be a crystallographic Tits arrangement with respect to $R$. Assume the same notation as in Lemma \ref{restmap}. Then $\sigma_{K,L}|_{B^{K}_x}$ is a bijection from $B^{K}_x$ to $B^{L}_x$ mapping $\alpha_i$ to $\beta_i$.
\label{restmap2}
\end{fol}

\begin{prop}
Let $(\AAA,T)$ be a crystallographic Tits arrangement w.r.t.\ $R$ and let $x \in \ol{T}$ with $R_x \neq \emptyset$. Then
$R_x \subset \sum_{\alpha \in B^{K}_x} \ZZ \alpha$ for all $K \in K_x$.
\label{indcryst}
\end{prop}
\begin{proof}
Let $K,L \in \KKK_x$ be adjacent chambers, $B^{K} = \{\alpha_1, \dots, \alpha_r\}$ and assume that $B^{L} = \{\beta_1, \dots, \beta_r\}$ is indexed compatibly with $B^K$. Assume w.l.o.g.\ $\overline{K} \cap \overline{L} \subset \alpha_1^\perp \cap \ol{K}$ and $B^{K}_x = \{\alpha_1, \dots, \alpha_m\}$, $B^{L}_x = \{\beta_1, \dots, \beta_m\}$ for some $1 \leq m \leq r$. We know by Corollary \ref{restmap2} that the mapping $\sigma_{K,L}|_{B^{K}_x}: B^{K}_x \rightarrow B^{L}_x, \alpha_i \mapsto c_i\alpha_1 + \alpha_i$ is a bijection. In particular, we get $B^{L}_x \subset \sum_{\alpha \in B^{K}_x} \ZZ \alpha$.

Now let $\alpha \in R_x$. Since $x \in \alpha^\perp$, $x$ is contained in some simplex, and therefore also in some maximal simplex. Thus there exists a chamber $L \in \KKK_x$, such that $r \in B^{L}_x$.
Using that $\KKK_x$ is connected \cite[Lemma 2.25]{CMW}, there exists a chain $K = K_0, K_1, \dots, K_{k-1}, K_k =d$ such that $K_{i-1}, K_i$ are adjacent. It follows by induction and using the fact that $\sigma_{K_{i-1}, K_i}$ maps roots to integral linear combinations of $B^{K_{i-1}}_x$ that $B^{K_{i}}_x \subset \sum_{j=1}^m \ZZ \alpha_i$ for all $0 \leq i \leq k$.
\end{proof}

\begin{fol}
Let $(\AAA, T)$ be a crystallographic Tits arrangement with respect to $R$, $x \in \ol{T}$ such that $R_x \neq \emptyset$. Then the Tits arrangement $(\AAA_x^\pi,T_x)$ is crystallographic with respect to $R_x$.
\label{cryst:globtoloc}
\end{fol}
\begin{proof}
This is a direct consequence from Proposition \ref{indcryst}, since the crystallographic property only depends on $R_x$.
\end{proof}

\begin{fol}
Let $(\AAA, T)$ be a crystallographic Tits arrangement with respect to the reduced root system $R$, $x \in \ol{T}$ such that $R_x \neq \emptyset$. Let $K \in \KKK_x$, $B^K_x = \{\alpha_i \mid i \in J\}$ and $\Cc = \Cc(\AAA, T,R)$. Then $\Cc(\AAA_x^\pi,T_x,R_x) = \Cc_J^K$. Likewise, if $J \subset I$, $a \in A$, $B^a = \{\alpha_i \mid i \in I\}$, we find $x \in \ol{T}$ such that $x \in \alpha_i^\perp$ if and only if $i \in J$. In this case $\Cc_J^a = \Cc(\AAA_x^\pi, T_x, R_x)$.

In other words, parabolic subarrangements of crystallographic arrangements correspond to residues of the respective Cartan graph.
\end{fol}
\begin{proof}
The corollary is immediate from the correspondence of Cartan graphs and Tits arrangements in Corollary \ref{cor:cor}.
\end{proof}

\begin{prop}
Let $(\AAA,T)$ be a Tits arrangement associated to $R$ of rank $r \neq 2$. If for all $0 \neq x \in \ol{T}$ the arrangements $(\AAA_x^\pi,T_x)$ are crystallographic with respect to $R_x$, then $(\AAA,T)$ is crystallographic with respect to $R$.

If $(\AAA,T)$ is locally spherical, then $(\AAA,T)$ is crystallographic with respect to $R$ if for all $0\neq x \in T$ the arrangements $(\AAA_x^\pi,T_x)$ are crystallographic with respect to $R_x$.
\label{cryst:loctoglob}
\end{prop}
\begin{proof}
If $r=1$ we can conclude that $T = \RR$ or $T  = \RR_{>0}$. In the first case the root system $R$ is 1-dimensional and crystallographic. In the second case there is no thin hyperplane arrangement, as every hyperplane is just $\{0\}$ and this does not intersect $\RR_{>0}$. Thus assume $r\geq 3$.

We know $\SSSS$ is connected by Proposition \ref{prop-Sproperties}, so it is enough to show that for two adjacent chambers $K,L \in \KKK$ we have $B^L \subset \sum_{\alpha \in B^K} \ZZ \alpha$. The proposition follows then by induction on the length of a minimal gallery between arbitrary chambers $K', L' \in \KKK$.

Assume further that $B^K = \{\alpha_1, \dots, \alpha_r\}$, $B^L = \{\beta_1, \dots, \beta_r\}$ are indexed compatibly with $B^K$ and $K,L$ are adjacent in $\alpha_1$. Take two distinct vertices $v_1, v_2$ of $\ol{K}$ such that $v_1, v_2 \in \alpha_1^\perp$. Note that if $(\AAA,T,R)$ is locally spherical, these vertices are contained in $T$. As $\ol{L}$ is a simplicial cone, there exists a unique maximal face not containing $v_1$, which must contain $v_2$. Therefore by Proposition~\ref{indcryst}
$$B^L \subset B^{L}_{v_1} \cup B^{L}_{v_2} \subset \sum_{\alpha \in B^{K}_{v_1}} \ZZ \alpha \cup \sum_{\alpha \in B^{K}_{v_2}} \ZZ \alpha \subset \sum_{\alpha \in B^K}\ZZ \alpha$$
and we are done.
\end{proof}

\begin{exmp}
\label{cryst:exmp1}
The requirement $r\neq 2$ in Proposition \ref{cryst:loctoglob} is actually necessary. Take the root system of type $\tilde{A_1}$, $R = \{e_i^\vee +k \gamma \mid i=1,2, k \in \ZZ \}$, where the $e_i^\vee$ are the standard base vectors and $\gamma = e_1^\vee + e_2^\vee$. Denote $v_\lambda = \lambda e_1 + (1-\lambda)e_2$ and take $\{v_\lambda \mid \lambda \in \ZZ\}$ as a vertex set, which is a lattice in the affine space $W = \{v \in \RR^2 \mid  \gamma(v) = 1\}$. An open simplex $s_\lambda$ is just the open convex hull of $v_{\lambda}$ and $v_{\lambda+1}$, and the chambers $K_\lambda$ are the respective cones $\RR_{>0}s_\lambda$ in $T = \gamma^+$.

Given the chambers as above, there are bases $B^{\lambda} := B^{K_\lambda}$ as $B^0 = \{e_1^\vee, e_2^\vee\}$ and 
\begin{align*}
B^n &= \{e_2^\vee+n\gamma, -(e_2^\vee+(n-1)\gamma)\},\\
B^{-n} &=\{e_1^\vee+n\gamma, -(e_1^\vee+(n-1)\gamma)\}
\end{align*}
where $n\in \NN$. Note that the order in which the elements of the sets are written down does not give a compatible indexing with each other. Now for $n\in \NN, m \in \NN_0$ the arrangements at the vertices are given by the roots
\begin{align*}R_{v_n} &= \{\pm (e_2^\vee + (n-1)\gamma)\},\\
R_{-m} &= \{\pm (e_1^\vee + m\gamma\}.
\end{align*}
The root system $R$ itself is crystallographic as well as the arrangements at the points $v_\lambda$. Also note that $R = \bigcup_{\lambda \in \ZZ} B^\lambda$.
One can modify $R$ by defining the root system
\begin{align*}
\tilde{R} = \{(k+1)(e_i^\vee+k\gamma)\mid i = 1,2, k \in \ZZ\}.
\end{align*}
Since $\tilde{R}$ only contains multiples of the elements in $R$, the arrangement induced by $R$ and by $\tilde{R}$ is actually the same, so we find the same set of chambers $K_\lambda$, which are induced by the same simplices on the same vertex set. Then the root system at the vertices are
\begin{align*}
\tilde{R}_{v_n} &= \{\pm n(e_2^\vee + (n-1)\gamma)\},\\
\tilde{R}_{v_{-m}} &= \{\pm (m+1)(e_1^\vee + m\gamma\}.
\end{align*}
This implies that for every point in $T$ the induced arrangement is crystallographic. Note that every point in $T$ different from the $v_\lambda$ is either a multiple of some $v_\lambda$ and therefore induces the same arrangement or is in the interior of a simplicial cone and induces the empty arrangement.

Now consider the root bases with respect to $\tilde{R}$, which we shall call $\tilde{B}$ to distinguish them from the sets $B^{\lambda}$. These are of the form
\begin{align*}
\tilde{B}^n &= \{(n+1)(e_2^\vee+n\gamma), -n(e_2^\vee+(n-1)\gamma)\},\\
\tilde{B}^{-n} &=\{e_1^\vee+n\gamma, -(e_1^\vee+(n-1)\gamma)\}.
\end{align*}
So $\tilde{B}^0 = \{e_1^\vee, e_2^\vee\}$, $\tilde{B}^1 = \{2e_1^\vee + 4e_2^\vee, -e_2^\vee\}$. Now $e_1^\vee = \frac{1}{2}(2e_1^\vee + 4e_2^\vee) +2 (-e_2^\vee)$, so $\tilde{R}$ does not satisfy the crystallographic property.
\end{exmp}

\begin{rem}
\begin{enumerate}
\item
We assume $x \neq 0$ in Proposition \ref{cryst:loctoglob} since $\AAA_0 = \AAA$.
\item For the proof of Proposition \ref{cryst:loctoglob} it is actually sufficient to assume the crystallographic property for all parabolic subarrangements $(\AAA_x,T_x)$ where $0\neq x$ is contained in some vertex $v$ in the simplicial complex $\SSSS$. It is also not hard to see that being crystallographic in such a point implies $R_y$ being crystallographic for all $y$ such that the minimal simplex $F_y$ containing $y$ is contained in $St(v)$. Thus $R_y$ is crystallographic for all $y \in \ol{T}$.
\end{enumerate}
\end{rem}

\subsection{Restrictions}

\begin{defi}
Let $(\AAA,T)$ be a Tits arrangement associated to $R$, $H \in \AAA$. Define
\begin{align*}
\AAA^H &:= \{H' \cap H \leq H\ |\ H' \in \AAA \setminus \{H\}, H' \cap H \cap T \neq \emptyset\},\\
\pi_H^\ast&: V^\ast \to H^\ast,\:\: \alpha \mapsto \alpha|_H,\\
R^H &:= \pi_H^\ast(R) \setminus (\{0\} \cup \{\alpha \in \pi_H^\ast(R) \mid \alpha^\perp \cap H \cap T = \emptyset\}).
\end{align*}
\end{defi}

The following are the main results from \cite{CMW} on restrictions of simplicial arrangements.

\begin{prop}[{\cite[Proposition 4.16]{CMW}}]\label{indarrange}
Let $(\AAA,T)$ be a $k$-spherical Tits arrangement for $k \geq 1$. Let $R$ be a root system associated to $(\AAA,T)$ and $H \in \AAA$. Then $(\AAA^H,T \cap H)$ is a $k-1$-spherical simplicial arrangement. If $(\AAA^H, T \cap H)$ is thin, it is a Tits arrangement associated to $R^H$.
\end{prop}

\begin{fol}[{\cite[Corollary 4.17]{CMW}}]\label{projbasis}
Assume that in the setting of Proposition \ref{indarrange} the pair $(\AAA^H, T \cap H)$ is a Tits arrangement and $R^H$ a root system associated to $(\AAA^H, T \cap H)$. Let $K \in \KKK$ such that $H \in W^H$. Then the set $\pi_H^\ast(B^K) \setminus\{0\}$ is a basis of $H^\ast$.
\end{fol}

\begin{defi}
The pair $(\AAA^H, T \cap H)$ is called the \textit{arrangement induced by $H$} or the \textit{restriction of $\AAA$ to $H$}.
\end{defi}

\begin{rem}
Given a Tits arrangement $(\AAA,T)$ associated to $R$, in general it seems that properties of $R$ are hard to transfer to restrictions. This is suggested by the root systems of $F_4$ or $\tilde{F_4}$ in the example below. Here we start with the strongest properties we have, i.e.\ a crystallographic reduced root system associated to a Weyl group. However, inducing on reflection hyperplanes yields in some cases only a root system associated to a non-standard Cartan graph, which is not reduced anymore. Nevertheless, the crystallographic property is inherited. We will dedicate the rest of this section to show that this is always the case for restrictions.
\end{rem}

From now on, assume that $(\AAA,T)$ is a crystallographic Tits arrangement associated to $R$. In this case, we can state a stronger version of the following lemma in \cite{CMW}.

\begin{lem}[{\cite[Lemma 4.15]{CMW}}]\label{indbases}
Let $(\AAA,T)$ be a Tits arrangement associated to $R$, $K \in \KKK$, $H \in W^K$. Let $B := \pi_H^\ast(B^K) \setminus \{0\}$. Then
\begin{enumerate}[label=\roman*)]
	\item $H \cap \overline{K}= \ol{K'}$ for a unique $K' \in \KKK^H$,
	\item $\langle \ol{K'} \cap \alpha^\perp \rangle = \alpha^\perp$  and $\alpha^\perp \cap K' = \emptyset$ for $\alpha \in B$.
	\item $K' = \{x \in H \mid \alpha(x) > 0 \text{ for all } \alpha \in B\}$.
\end{enumerate}
\end{lem}

\begin{prop}
Let $(\AAA,T)$ be a $2$-spherical crystallographic Tits arrangement with respect to $R$ and $H \in \AAA$. Let $K \in \KKK$ such that $H \in W^K$ and $K' \in \KKK^H$ such that $\ol{K'} = \ol{K} \cap H$. Then 
\begin{enumerate}[label=\roman*)]
	\item $W^{K'} = \{\alpha^\perp \ |\ \alpha \in \pi_H^\ast (B^K) \setminus \{0\}\}$,
	\item For $\alpha \in \pi_H^\ast (B^K) \setminus \{0\}$ we have that $\beta \in R^H \cap \langle \alpha \rangle$
	implies $\beta = \lambda \alpha$, $\lambda \in \ZZ$,
	\item $\pi_H^\ast (B^K) \setminus \{0\} \subset (R^H)^{red}$,
	\item $\pi_H^\ast (B^K) \setminus \{0\} = B^{K'}$.
\end{enumerate}
\label{Rhbasic}
\end{prop}
\begin{proof}
Assertion i) is an immediate consequence of Lemma \ref{indbases} ii).

To check ii), let $\alpha_1 \in B^K$ such that $\alpha_1^\perp = H$, let $\alpha \in R$ and $R_\alpha := \{\beta \in R \ |\ \beta^\perp \cap H = \alpha^\perp \cap H \} = \{\beta \in R\ |\ \pi_H^\ast(\beta) \in \langle \pi_H^\ast(\alpha) \rangle \setminus \{0\}\}$ as in the proof of \cite[Proposition 4.16]{CMW}. Since $(\AAA,T)$ is 2-spherical, there exists $x \in \alpha^\perp \cap H \cap T$. Consider the arrangement $(\AAA_x, T_x)$. This is a crystallographic arrangement with respect to $R_x$ by Proposition \ref{cryst:globtoloc}, as $(\AAA,T)$ is crystallographic with respect to $R$. In this arrangement consider the linearly independent set $B^{K}_x$, then $\alpha_1, \alpha \in B^{K}_x$. Let $B^{K}_x = \{\alpha_1, \alpha, \tau_3, \dots, \tau_m\}$ for some $m \in \NN$. Now for an arbitrary $\beta \in R_\alpha$, we know $\beta \in R_x$ since $x \in \alpha^\perp \cap H = \beta^\perp \cap H$. Therefore $\beta = \lambda_1 \alpha_1 + \lambda_2 \alpha + \sum_{i=3}^m \lambda_i \tau_i$ with $\lambda_i \in \ZZ$ for $i = 1, 2, \dots, m$ either all positive or negative.

By Corollary \ref{projbasis}, $\pi_H^\ast(B^K) \setminus \{0\}$ is a linearly independent set.
Applying $\pi_H^\ast$ yields
$$\pi_H^\ast(\beta) = \lambda_2 \pi_H^\ast(\alpha) + \sum_{i=3}^m \lambda_i \pi_H^\ast(\tau_i),$$
and thus by the choice of $\beta$ we get $\pi_H^\ast(\beta) = \lambda_2 \pi_H^\ast(\alpha)$ with $\lambda_2 \in \ZZ$, as desired. This shows ii).

Assertion iii) is a consequence of ii), with respect to the standard reductor, which also exists due to ii). Finally, iv) is immediate from iii) and Lemma \ref{indbases}.
\end{proof}

\begin{prop}
Let $(\AAA,T)$ be a $2$-spherical crystallographic Tits arrangement with respect to $R$ and $H \in \AAA$. Then $(\AAA^H, T \cap H)$ is a crystallographic Tits arrangement with respect to $(R^H)^{red}$.
\label{indrootsys}
\end{prop}
\begin{proof}
Let $H = \alpha_1^\perp$ for some $\alpha_1 \in R^{red}$. Let $K' \in \KKK_H$ and let $K \in \KKK$ such that $\overline{K'} \subset \overline{K}$. By Proposition \ref{Rhbasic} iv) we know $B^{K} = \{\alpha_1, \dots, \alpha_r\}$ with $B^{K'}= \{\pi_H^\ast(\alpha_i) \mid i=2, \dots r\}$. Take $\beta_H \in R^H$, $\beta \in R$ such that $\pi_H^\ast(\beta) = \beta_H$. Since $(A,T,R)$ is crystallographic, $\beta = \sum_{i=1}^{n} \lambda_i \alpha_i$ with $\lambda_i \in \ZZ$. Therefore we get
$\pi_H^\ast(\beta) = \sum_{i=1}^{r} \lambda_i \pi_H^\ast(\alpha_i) = \sum_{i=2}^{n} \lambda_i \pi_H^\ast(\alpha_i)$,
so $R^H$ is indeed crystallographic.
\end{proof}

\begin{rem}
\begin{enumerate}
	\item In the case where $(\AAA,T)$ is not $2$-spherical, the restriction $(\AAA^H, T \cap H)$ might be thin nonetheless, and in this case $(A^H,T \cap H)$ is again a crystallographic Tits arrangement with respect to $R^H$.
	\item Proposition \ref{Rhbasic} also yields that to obtain a reduced root system for $(R^H)$ it is sufficient to consider the $\pi_H^\ast(B^K)$ for all chambers $K$ with $H \in W^K$. In other words,
$$(R^H)^{red} = \bigcup_{H \in W^K} \pi_H^\ast(B^K) \setminus \{0\}.
$$
\end{enumerate}

\end{rem}

\begin{exmp}
\label{exmpF4}
The property of being reduced is not inherited by $R^H$. Also, if $R$ is a root system, $R^H$ constructed in the way above does not need to be a root system as well, as the following example shows. Take the root system of $\tilde{F_4}$ (which certainly is reduced). Let $\iota: \RR^4 \to (\RR^4)^\ast, v \mapsto (v,\cdot)$ be the standard isomorphism, where $(\cdot,\cdot)$ is the standard inner product. Note that $\iota$ takes the standard basis $\{e_1, e_2, e_3, e_4\}$ to its dual $\{e_1^\vee, e_2^\vee,e_3^\vee,e_4^\vee\}$. We will denote elements of $(\RR^4)^\ast$ as vectors with respect to this basis.

A set of simple roots for $F_4$ in $\RR^4$ is for example (cp.\ \cite{Bo02}):
$$\left\{\varphi_1=(0,1,-1,0),
\varphi_2=(0,0,1,-1),
\varphi_3=(0,0,0,1),
\varphi_4=\frac{1}{2}(1,-1,-1,-1)\right\}.$$
Then the whole root system $R(F_4)$ can be described as $\iota$ of
\begin{enumerate}[label=\roman*)]
	\item vectors with two components $1$ or $-1$, $0$ otherwise,
	\item vectors with one component $1$ or $-1$, 0 otherwise,
	\item vectors with all four components $\frac{1}{2}$ or $-\frac{1}{2}$.
\end{enumerate}
So there are 24 roots of type i), 8 of type ii) and 16 of type iii). We will compute orthogonal projections of $R(F_4)$ on two simple roots. For $1 \leq i \neq j \leq 4$ let $\pi_{ij}: (\RR^4)^\ast \to (\varphi_i^\perp \cap \varphi_j^\perp)^\ast$ denote the respective restriction. The respective projections are:
\begin{align*}
\pi_{12}(R(F_4)) =
 &\left\{ \pm 
(1,0,0,0),\pm
(0,\frac{1}{3},\frac{1}{3},\frac{1}{3}),\pm
(1,\frac{1}{3},\frac{1}{3},\frac{1}{3}), \pm
(-1,\frac{1}{3},\frac{1}{3},\frac{1}{3}),
\pm
(\frac{1}{2},\frac{1}{2},\frac{1}{2},\frac{1}{2}), \right.\\
&\left. \pm 
(-\frac{1}{2},\frac{1}{2},\frac{1}{2},\frac{1}{2}),  \pm
(0,\frac{2}{3},\frac{2}{3},\frac{2}{3}), \pm
(\frac{1}{2},\frac{1}{6},\frac{1}{6},\frac{1}{6}), \pm
(-\frac{1}{2},\frac{1}{6},\frac{1}{6},\frac{1}{6}), 0
\right\},
\\
\pi_{13}(R(F_4)) =
 &\left\{ \pm 
(1,0,0,0),\pm 
(0,1,1,0),\pm
(0,\frac{1}{2},\frac{1}{2},0),\pm
(1,\frac{1}{2},\frac{1}{2},0), \pm
(-1,\frac{1}{2},\frac{1}{2},0),
\right.\\&\left.
\pm
(\frac{1}{2},\frac{1}{2},\frac{1}{2},0),
\pm (-\frac{1}{2},\frac{1}{2},\frac{1}{2},0),\pm 
(-\frac{1}{2},0,0,0),  0
\right\},
\\
\pi_{14}(R(F_4)) =
 &\left\{
\pm (\frac{3}{4},\frac{1}{4},\frac{1}{4},\frac{1}{4}),
\pm (\frac{1}{4},\frac{1}{4},\frac{1}{4},-\frac{1}{4}),
\pm (\frac{1}{4},-\frac{1}{4},-\frac{1}{4},\frac{3}{4}),
\pm (1,\frac{1}{2},\frac{1}{2},0),
\pm (\frac{1}{2},0,0,\frac{1}{2}), \right.
\\
&\left.
\pm (\frac{1}{2},\frac{1}{2},\frac{1}{2},-\frac{1}{2}),
\pm (1,0,0,1),
\pm (0,\frac{1}{2},\frac{1}{2},-1)
,0 \right\},
\\
\pi_{23}(R(F_4)) =
 &\left\{ \pm 
(1,0,0,0),\pm
(0,1,0,0),\pm
(1,1,0,0),
\pm
(1,-1,0,0), \pm
(\frac{1}{2},\frac{1}{2},0,0),
\pm
(\frac{1}{2},-\frac{1}{2},0,0), 0
\right\}, 
\end{align*}
\begin{align*}
\pi_{24}(R(F_4)) =
 &\left\{ \pm 
(\frac{1}{2},-\frac{1}{2},0,0),\pm 
(0,0,\frac{1}{2},\frac{1}{2}), \pm 
(1,-1,0,0), 
\right.\\&\left.
\pm 
(\frac{1}{2},-\frac{1}{2},\frac{1}{2},\frac{1}{2}), \pm 
(\frac{1}{2},-\frac{1}{2},-\frac{1}{2},-\frac{1}{2}), \pm 
(0,0,1,1), 
0\right\}, 
\\
\pi_{34}(R(F_4)) = &\left\{\pm 
(\frac{2}{3},-\frac{1}{3},-\frac{1}{3},0),\pm 
(-\frac{1}{3},\frac{2}{3},-\frac{1}{3},0), \pm 
(-\frac{1}{3},-\frac{1}{3},\frac{2}{3},0), 
\right.\\&\left.
\pm 
(1,-1,0,0), \pm 
(1,0,-1,0), \pm 
(0,1,-1,0), 0
\right\}. 
\end{align*}
Let $R_{ij}:=\pi_{ij}(R(F_4)) \setminus \{0\}$. These are non-reduced crystallographic rank two root systems. After reducing, consider the images of the two remaining elements of the simple roots. Writing the roots as linear combinations of the two yields:
\begin{itemize}
	\item $R_{23}$ and $R_{24}$ are combinatorially equivalent to $B_2$.
	\item $R_{12}$ and $R_{34}$ are combinatorially equivalent to $G_2$.
	\item $R_{13}$ and $R_{14}$ are combinatorially equivalent to $R(1,2,2,2,1,4)$.
\end{itemize}
Here $R(1,2,2,2,1,4)$ denotes the rank two root system associated to the sequence $(1,2,2,2,1,4)$ according to the classification of spherical rank two Weyl groupoids \cite{CH09}.
\end{exmp}


\providecommand{\bysame}{\leavevmode\hbox to3em{\hrulefill}\thinspace}
\providecommand{\MR}{\relax\ifhmode\unskip\space\fi MR }
\providecommand{\MRhref}[2]{%
  \href{http://www.ams.org/mathscinet-getitem?mr=#1}{#2}
}
\providecommand{\href}[2]{#2}

\end{document}